\documentclass[12pt,reqno,psfig,openbib]{amsart}

\usepackage[pagebackref=true, colorlinks=true, citecolor=blue]{hyperref}
\usepackage{amssymb,amsmath,graphicx,amsfonts,euscript}
\usepackage{fancyhdr}
\usepackage{epsfig}
\usepackage{setspace}
\usepackage{dsfont}
\usepackage[all]{xy}
\usepackage{amsmath}
\usepackage{enumitem}

\usepackage{caption}
\usepackage{subcaption}

\usepackage{soul}

\usepackage{float}

 \usepackage{pst-grad} 
 \usepackage{pst-plot} 

\newtheorem{theorem}{Theorem}[section]
\newtheorem{rem}[theorem]{Remark}

\numberwithin{equation}{section}
\numberwithin{figure}{section}

 \usepackage[left=1.25in,top=1.15in,right=1.25in]{geometry}
\usepackage{verbatim}

\newcommand{\spacer}{\vspace{1.5mm}}

\date{\today}

\renewcommand{\epsilon}{\varepsilon}
\newcommand{\ep}{\varepsilon}

\newcommand{\ov}{\overline}
\newcommand{\pa}{\partial}
\newcommand{\tht}{\theta}
\newcommand{\Tht}{\Theta}

\newcommand{\pri}{\prime}
\newcommand{\sig}{\sigma}
\newcommand{\kap}{\kappa}
\newcommand{\blds}{\boldsymbol}

\begin{document}

\title[SL-PINN methods for boundary layer problems]
{Semi-analytic PINN methods for boundary layer problems in a rectangular domain}

\author[G.-M. Gie, Y. Hong, C.-Y.  Jung, and T. Munkhjin]
{Gung-Min Gie$^{1}$, 
Youngjoon Hong$^{2}$,  
Chang-Yeol  Jung$^3$, and
Tselmuun Munkhjin$^3$}
\address{$^1$ Department of Mathematics, University of Louisville, Louisville, KY 40292}
\address{$^2$ Department of Mathematical Sciences, Korea Advanced Institute of Science and Technology}
\address{$^3$ Department of Mathematical Sciences, Ulsan National Institute of Science and Technology, 
Ulsan 44919, Korea}
\email{gungmin.gie@louisville.edu}
\email{hongyj@kaist.ac.kr}
\email{cjung@unist.ac.kr}
\email{tmunkhjin@unist.ac.kr}

\begin{abstract}
Singularly perturbed boundary value problems pose a significant challenge for their numerical approximations because of 
the presence of sharp boundary layers. 
These sharp boundary layers are responsible for the stiffness of solutions, which leads to large computational errors, if not properly handled. 
It is well-known that the 
classical numerical methods as well as 
the Physics-Informed Neural Networks (PINNs)  
require some special treatments near the boundary, e.g., using 
extensive mesh refinements or 
finer collocation points, 
in order to obtain an accurate approximate solution especially inside of the stiff boundary layer. 
In this article, 
we modify the PINNs  
and construct our new semi-analytic SL-PINNs suitable for singularly perturbed boundary value problems. 
Performing the boundary layer analysis, we first find the corrector functions describing the singular behavior of the stiff solutions inside boundary layers. 
Then  
we obtain the SL-PINN approximations of the singularly perturbed problems 
by embedding the explicit correctors in the structure of PINNs or 
by training the correctors together with the PINN approximations.  
Our numerical experiments confirm that our new SL-PINN methods produce
stable and accurate approximations for stiff solutions.
\end{abstract}


\maketitle

\tableofcontents

\section{Introduction}\label{S. Intro}

We employ the neural network architectures to approximate effectively the solutions of various 2D singularly perturbed elliptic boundary value problems in a rectangular domain $\Omega=(0,1)^2$. 
More precisely, we consider,  
\begin{align}\label{con_dif}
        \left\{
                \begin{array}{rl}
                	\spacer       
        - \ep\Delta u - \mathbf{b}\cdot\nabla u 
        = f, &\ \text{in $\Omega$},\\
        u = 0, &\ \text{on $\partial\Omega$},
        \end{array}
        \right.
\end{align}
where $0<\ep<\!< 1$ is a small perturbation parameter. 
The
given functions 
$\mathbf{b} = (b_1(x,y),b_2(x,y))$ 
and 
$f=f(x,y)$ are assumed to be smooth as much as needed for the analysis and computations performed in this article.

Singularly perturbed boundary value problems, e.g., the one described by equation (\ref{con_dif}) above, 
are known to generate the so-called  boundary layer, which is a thin region near the boundary of the domain where the solution exhibits a sharp transition. 
The mathematical theory of singular perturbations and boundary layers has been extensively studied in the literature, and several books and papers have been written on the subject, including \cite{Book16, Ho95, OM08, SK87}.

From the scientific computation's  point of view, 
approximating stiff solutions to any singular perturbation problem is well-known to be a challenging task  because the stiffness of the solution, especially inside the boundary layer, causes a large computational error. 
Most of the traditional numerical methods, such as finite element or finite volume methods, often require massive mesh refinements near the boundary in order to achieve a sufficient accuracy of numerical solutions, but the computational cost for this type of refinements is highly expensive with respect to the size of the small perturbation parameter.

Compared to the traditional methods, some semi-analytical methods have been proposed to compute the solutions more efficiently and effectively, such as the ones presented in \cite{CJL19, GJL1, GJL2, GJL22, HK82, HJL13, HJT14}. 
These methods enrich the basis of traditional numerical methods by adding a global basis function, called \textit{correctors}, which describes the singular behavior of the solution inside the boundary layer. 
By enriching the basis, these methods can capture the sharp transition of the solution without requiring massive mesh refinements near the boundary. 
The semi-analytic methods have proven to be highly efficient in solving singularly perturbed problems, and they have been successfully applied in various applications. 

In our paper, we use the  neural network methods, in particular the Physics-Informed Neural Networks (PINNs) to solve the singularly perturbed problems. 
Neural networks offer a novel approach to approximating solutions to differential equations using a compositional structure instead of an additive one. 
This approach allows for analytical expression of the solution, eliminating the need for interpolation. The formulation of the loss function is straightforward, and the additional effort required for problem-dependent factors is minimal. The method is mesh-free and generally applicable \cite{BN18, BE21, CCWX22, HL21, KKLPWY21, KZK19, KZK21, RK18, RBPK17, SS18, FSF22}. However, there is still a lack of conclusive results regarding convergence speed and approximation accuracy.

Raissi et al.'s PINNs \cite{RPK19} drew inspiration from earlier works by e.g. Lagaris et al. \cite{LLF98} and they expanded on existing concepts and introduced fundamentally new approaches, including a discrete time-stepping scheme that optimizes the predictive power of neural networks. However, while PINNs have shown promising results in various applications, such as fluid dynamics and solid mechanics, they are known to suffer from the spectral bias phenomenon \cite{CF21, RB19}.
The spectral bias phenomenon refers to the tendency of neural networks to learn low-frequency features before high-frequency features. In other words, the network tends to learn coarse features before fine details, which can result in a biased representation of the solutions. It is also known that the PINNs often fail to solve the solutions exhibiting multi-scale structures \cite{LA}.
The PINNs are particularly problematic when dealing with singularly perturbed problems, where the solution can vary rapidly over small scales in the domain, i.e., high-frequency solutions. In these cases, the neural network may not have sufficient resolution or detail to accurately capture the solution, leading to inaccurate results \cite{FT20, RYK20, WTP21}. Recently, the boundary-layer PINNs are proposed
to utilize the matching techniques of the outer and inner layer solutions based on singular perturbation theory \cite{ACD}.

To overcome the spectral bias phenomenon or to solve rapidly varying solutions, 
we combine the semi-analytic methods with the PINNs (e.g., see \cite{GHJ}). 
In fact, 
we obtain our new SL-PINN (singular layer PINN) approximations for the singularly perturbed problems 
by embedding the explicit correctors in the structure of PINNs or 
by training the correctors together with the PINN approximations. 
For this process, in Section \ref{bl_analysis}, we first perform the boundary layer analysis for the problem (\ref{con_dif}), and study the asymptotic behavior of the solution.  
In Section \ref{sec_numerics}, we use the analytic results to modify the PINNs and 
construct our new SL-PINNs suitable for 
the singular perturbation problem (\ref{con_dif}). 
Numerical computations of the SL-PINNs are presented as well to verify that our novel SL-PINNs capture well the sharp transition of the stiff solution and hence produce a good approximation.

\section{Boundary layer analysis}\label{bl_analysis}
In this section, 
we investigate 
the boundary layer 
of  
the singularly perturbed convection-diffusion 
problem(\ref{con_dif}) in a rectangular domain with corners.  
Depending on the direction of the convection vector $\blds{b}$, 
the solution to (\ref{con_dif}) exhibits a distinct behavior 
near each boundary and corner of the rectangular domain $\Omega$.  
Hence we consider two different cases in this article:
\begin{itemize}
    \item[(1)] 
        $b_1(x,y)>0$ and $b_2(x,y) > 0$ in $\Omega$ (\textit{Non-characteristic boundary case})
    \item[(2)] 
        $b_1 > 0$ and $b_2 = 0$ in $\Omega$ (\textit{Characteristic boundary case})
\end{itemize}

Especially when a portion of the boundary is characteristic, i.e., 
the convection vector $\blds{b}$ is parallel to that portion of the boundary like the case $2$ above,  
the boundary layers of (\ref{con_dif}) associated with the two sides sharing the corner 
may interact with each other as well as the corner singularity;  
see, e.g., \cite{Gri85, SK87, Book16, GJT13, GJT_Review} that demonstrate  the diversity and complexity of the boundary layers in a rectangle. 


For the boundary layer analysis below, 
we introduce and use 
the following norms,
\begin{align}
\|u\|_{L^2(\Omega)} = \left( \int_{\Omega} |u|^2 dxdy \right)^{1/2},
\end{align}
\begin{align}\label{energy}
\|u\|_\ep = \|u\|_{L^2(\Omega)} + \sqrt{\ep}\|\nabla u\|_{L^2(\Omega)}.
\end{align}

\subsection{Non-characteristic boundary case  ($b_1>0, \, b_2>0$)}\label{sec_non_char}
When the boundary is non-characteristic with  
$b_1>0, \, b_2>0$, 
the flow associated with (\ref{con_dif}) 
is moving in across the boundaries at $x= 1$ and $y = 1$. 
Hence, 
by setting 
$\ep=0$ in (\ref{con_dif}) 
and imposing the so-called {\it inflow boundary condition} at $x= 1$ and $y = 1$, 
we find the equation for the limit solution $u^0$ as 
\begin{align}\label{con_dif_limit}
        \left\{
                \begin{array}{rl}
                	\spacer        
                        - b_1 \dfrac{\partial u^0}{\partial x} 
                        - b_2 \dfrac{\partial u^0}{\partial y} = f, &\ \text{in $\Omega$},\\
                        u = 0, &\ \text{at $x=1$ or $y=1$}.
                \end{array}
        \right.
\end{align}
The well-posedness and regularity results for the problem (\ref{con_dif_limit}) can be verified by using the method of characteristics. 
We assume hereafter that the data $\blds{b}$ and $f$ are smooth enough so that 
the limit solution $u^0$ is sufficiently regular for the boundary layer analysis below.

Now, to study the asymptotic behavior of $u^\ep$ (solution of (\ref{con_dif})) as $\ep \rightarrow 0$, 
we first notice that $ u - u^0 = 0$ on $x = 1$ or  $y =1$, but 
\begin{equation}\label{e:nonzero_boundary_non-chatacteristic}
    u-u^0 
    \neq 0 
    \text{ along the boundaries at $x = 0$ and $y = 0$.}
\end{equation}
Hence we expect that the so-called {\it ordinary  
boundary layers} occur near 
$x = 0$ and $y = 0$ (see Fig. \ref{f:correctors_nonchar} below).  
To resolve this inconsistency, 
we 
 construct sequentially below  
a number of correctors as  
\begin{itemize}
    \item [$u^0$:]
        corresponding limit solution defined in (\ref{con_dif_limit}),  
    \item [${\tht}_B$:]
        ordinary boundary layer corrector (OBL) near the bottom boundary at $y = 0$,  
    \item [${\tht}_L$:]
        ordinary boundary layer corrector (OBL) near the left boundary at $x = 0$,
    \item [$\eta$ :]
        corner boundary layer corrector (CBL) that 
        manages the interaction between 
        ${\tht}_B$ and ${\tht}_L$ near the corner at $(0, 0)$.
\end{itemize}

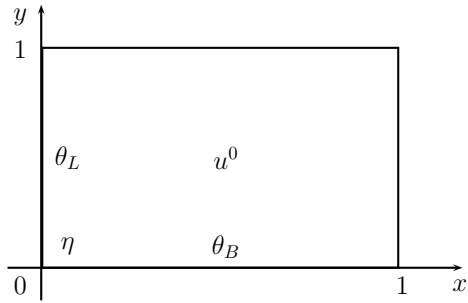
\begin{figure}[h] 
\begin{center}
\scalebox{.70} 
{
\begin{pspicture}(0,-3.16)(8.716875,3.16)
\psline[linewidth=0.04cm,arrowsize=0.05291667cm
2.0,arrowlength=1.4,arrowinset=0.4]{->}(0.056875,-1.86)(8.696875,-1.86)
\psline[linewidth=0.04cm,arrowsize=0.05291667cm
2.0,arrowlength=1.4,arrowinset=0.4]{->}(0.696875,-2.48)(0.696875,3.14)
\psframe[linewidth=0.04,dimen=outer](7.496875,2.34)(0.696875,-1.88)
\usefont{T1}{ptm}{m}{n}
\rput(4.2042187,0.225){\large ${u}^0$}
\usefont{T1}{ptm}{m}{n}
\rput(4.2042187,-1.45){\large ${\tht}_B$}
\usefont{T1}{ptm}{m}{n}
\rput(1.2,-1.45){\large $\eta$}
\usefont{T1}{ptm}{m}{n}
\rput(1.2,0.225){\large ${\tht}_L$}
\usefont{T1}{ptm}{m}{n}
\rput(0.3,-2.2){\large $0$}
\usefont{T1}{ptm}{m}{n}
\rput(0.3, 2.3){\large $1$}
\usefont{T1}{ptm}{m}{n}
\rput(0.3,2.9){\large $y$}
\usefont{T1}{ptm}{m}{n}
\rput(7.555,-2.2){\large $1$}
\usefont{T1}{ptm}{m}{n}
\rput(8.65,-2.2){\large $x$}
\end{pspicture}
}
\caption{Limit solution, and ordinary and corner boundary layer correctors}\label{f:correctors_nonchar}
\end{center}
\end{figure} 

We first introduce smooth cut-off functions near the left and bottom boundaries in the form, 
\begin{equation}\label{e:delta}
\begin{array}{cc}
\sig_B(y)
        = \left\{
                \begin{array}{cl}
                                1, & \text{$0 \leq y \leq 1/8$},\\
                                0, & \text{$y \geq 1/4$},
                \end{array}
            \right. 
&
\sig_L(x)
        = \left\{
                \begin{array}{cl}
                                1, & \text{$0 \leq x \leq 1/8$},\\
                                0, & \text{$x \geq 1/4$}.
                \end{array}
            \right.
\end{array}
\end{equation}

To resolve the discrepancies of $u-u^0$ at $y=0$, 
we write an asymptotic equation for the ordinary boundary layer near $y=0$,  
\begin{align}\label{obl_star_zero_B}
        \left\{
                \begin{array}{rl}
                	\spacer
		-\ep \dfrac{\partial^2 { \ov{\theta}_B}}{\partial y^2} 
        - b_2 (x, 0) \dfrac{\partial  \ov{ \theta}_B}{\partial y} = 0,
                    & \\
                \spacer
                 { \ov{\theta}_B} = - u^0(x,0),
                 &
                 \quad
                 \text{at $y=0$},\\
                 { \ov{\theta}_B} \rightarrow 0,
                 &
                 \quad
                 \text{as }  {y} \rightarrow \infty.
                \end{array}
        \right.
\end{align}
The explicit expression of $ { \ov{\theta}_B}$ is readily available as
\begin{align}\label{theta0_explicit_bar}
     { \ov{\theta}_B} 
        = - u^0(x,0) \exp\left(- b_2 (x, 0) \frac{y}{\ep}\right).
\end{align}
Then we define the ordinary boundary layer corrector 
${{\theta}_B}$ in the form, 
\begin{align}\label{theta_B}
     {{\theta}_B} 
        = 
            \sig_B(y) \ov{\theta}_B.
\end{align}
Using (\ref{obl_star_zero_B}) - (\ref{theta_B}), 
we find that the ordinary boundary layer corrector 
${{\theta}_B}$ satisfies the equation, 
\begin{align}\label{tht_B_eqn}
        \left\{
                \begin{array}{rl}
                	\spacer
		-\ep \dfrac{\partial^2 { {\theta}_B}}{\partial y^2} 
        - b_2 (x, 0) \dfrac{\partial  { \theta}_B}{\partial y} 
        &
        = 
                        -\ep \sig_B^{\pri} \dfrac{\pa \ov{\tht}_B}{\pa y}   
                        -\ep \sig_B^{\pri \pri} \ov{\tht}_B
                        -b_2(x, 0) \sig_B^{\pri} \ov{\tht}_B
                    = 
                    e.s.t.\footnote ,
                    \\
                \spacer
                 {{\theta}_B} 
                 &
                 = - u^0(x,0),
                 \quad
                 \text{at $y=0$},\\
                 \spacer
                 {{\theta}_B} 
                 &
                 = - \sig_B(y) u^0(0,0)  \exp\left(- b_2 (0, 0)y/\ep \right),
                 \quad
                 \text{at $x=0$},\\
                 {{\theta}_B} 
                 &
                 = 0,
                 \quad
                 \text{at $x=1$ or $y=1$}.
                \end{array}
        \right.
\end{align}
\footnotetext{$e.s.t.$ denotes a term that is exponentially small with respect to $\ep$ in the energy norm (\ref{energy}). An $e.s.t.$ is negligible compared to any $\ep^j$, $j \geq 0$.}

Similar to the process above near the boundary at $y=0$, 
to balance the non-zero value $u-u^0$ at $x=0$,
we introduce 
the ordinary boundary layer corrector 
${{\theta}_L}$ in the form, 
\begin{align}\label{theta_L}
     {{\theta}_L} 
        = 
            \sig_L(x) \ov{\theta}_L, 
\end{align}
where the $\ov{\theta}_L$ is defined as a solution of 
\begin{align}\label{obl_star_zero}
        \left\{
                \begin{array}{rl}
                	\spacer
		-\ep \dfrac{\partial^2 {\ov{\theta}_L}}{\partial x^2} 
        - 
        b_1 (0, y) 
        \dfrac{\partial  {\ov{\theta}_L}}{\partial x} = 0,
                    & \\
                \spacer
                 {\ov{\theta}_L} = - u(0,y),
                 &
                 \quad
                 \text{at $x=0$},\\
                 {\ov{\theta}_L} \rightarrow 0,
                 &
                 \quad
                 \text{as }  {x} \rightarrow \infty.
                \end{array}
        \right.
\end{align}
The explicit expression of $\ov{\theta}_L$ is given by  
\begin{align}\label{theta0_explicit_bar}
     {\ov{\theta}_L} 
        = - u^0(0,y) \exp\left(- b_1 (0, y) \frac{x}{\ep}\right).
\end{align}
One can verify that 
the ordinary boundary layer corrector 
${{\theta}_L}$ satisfies the equation, 
\begin{align}\label{tht_L_eqn}
        \left\{
                \begin{array}{rl}
                	\spacer
		-\ep \dfrac{\partial^2 { {\theta}_L}}{\partial x^2} 
        - b_1 (0, y) \dfrac{\partial  { \theta}_L}{\partial x} 
        &
        = 
                        -\ep \sig_L^{\pri} \dfrac{\pa \ov{\tht}_L}{\pa x}   
                        -\ep \sig_L^{\pri \pri} \ov{\tht}_L
                        -b_1(0, y) \sig_L^{\pri} \ov{\tht}_L
                    = 
                    e.s.t.,
                    \\
                \spacer
                 {{\theta}_L} 
                 &
                 = - u^0(0,y),
                 \quad
                 \text{at $x=0$},\\
                 \spacer
                 {{\theta}_L} 
                 &
                 = - \sig_L(x)u^0(0,0) \exp\left(- b_1 (0, 0)x/\ep \right),
                 \quad
                 \text{at $y=0$},\\
                 {{\theta}_L} 
                 &
                 = 0,
                 \quad
                 \text{at $x=1$ or $y=1$}.
                \end{array}
        \right.
\end{align}

At this point,
the difference between $u$
and the proposed expansion $u^0 +  \tht_B +  {\theta}_L$ 
attains the boundary values as 
\begin{equation}\label{e:CH4_2:GIE_2}
\begin{array}{l}
 \displaystyle
        u^{\ep} - (u^0 +  \tht_B +  {\theta}_L)
        			=
                        \left\{
                                \begin{array}{l}
                                		\spacer
                                         \displaystyle
                                                    - \tht_B(0, y)
                                                    =
                                                    \sig_B(y)
                    u^0(0,0)
                    \exp\Big(- b_2(0, 0)\frac{y}{\ep}\Big),
                                                        \quad
                                                    \text{at } x=0,\\
                                           \spacer
                                         \displaystyle
                                                     - {\theta_L (x, 0)}
                                                     =
                                                     \sig_L(x) 
                    u^0(0,0)
                    \exp\Big(- b_1(0, 0)\frac{x}{\ep}\Big),
                                                    \quad
                                                    \text{at } y=0,\\
                                        \displaystyle
                                                    0,
                                                    \quad
                                                    \text{at } y=1
                                                    \text{ or }
                                                    x = 1.
                                \end{array}
                        \right.
\end{array}
\end{equation}

To handle the non-zero boundary values 
$- \tht_B$ at $x=0$ and $- {{\tht}_L}$ at $y=0$, appearing in (\ref{e:CH4_2:GIE_2}),
we  introduce a corner corrector $ {{\eta}}$ in the form, 
\begin{align}\label{zeta_explicit_bar_zero}
             {\eta} 
                = 
                    \sig_L(x) \sig_B(y)
                    \ov{\eta}.
\end{align}
with 
\begin{align}\label{zeta_explicit_bar_zero_approx}
             \ov{\eta} 
                = 
                    u^0(0,0)
                    \exp\Big(- b_1(0, 0)\frac{x}{\ep}\Big)
                    \exp\Big(- b_2(0, 0)\frac{y}{\ep}\Big).
\end{align}

The corner corrector $\eta$ satisfies the equation, 
\begin{equation}\label{e:ocl_eq_zero}
        - \ep \dfrac{\partial^2 {{\eta}}}{\partial x^2} - \ep \dfrac{\partial^2 {{\eta}}}{\partial y^2}
        - b_1 (0, 0) \dfrac{\partial {{\eta}}}{\partial x} 
        - b_2 (0, 0) \dfrac{\partial {{\eta}}}{\partial y}
        = 
            e.s.t.
\end{equation}
Moreover, we 
see that the difference between $u$
and the proposed expansion, 
\begin{equation}\label{e:asymp_exp_non}
    u 
        \simeq 
            u^0 +  \tht_B +  {\theta}_L +  {\eta},
\end{equation}
is now finally balanced on the boundary, i.e.,  
\begin{equation}\label{e:CH4_2:GIE_3}
        u  - (u^0 +  \tht_B +  {\theta}_L +  {\eta})
        			=
			0,
			\quad
			\text{on }
			\pa \Omega.
\end{equation}

In order to validate our asymptotic expansion of $u$, 
we introduce
the difference between $u$ 
and the proposed expansion,
\begin{equation}\label{e:CH4_2_w0_eq}
    w
		=
		u
		-
		(u^0 + \tht_B +  {\theta}_L +  {\eta}).
\end{equation}
 
Using the equations (\ref{con_dif}), (\ref{con_dif_limit}), (\ref{tht_B_eqn}), (\ref{tht_L_eqn}), (\ref{e:ocl_eq_zero}), and 
(\ref{e:CH4_2:GIE_3}), 
we
write the equation for $w$,
\begin{equation}\label{e:w_eq_zero}
        \left\{
                \begin{array}{rl}
                                - \ep \Delta w
                                - b_1\dfrac{\pa w}{\pa x} 
                                - b_2 \dfrac{\pa w}{\pa y}
                                        =
                                                \mathcal{R}_1+
                                              \mathcal{R}_2  + e.s.t.
                               &
                               \quad
                               \text{in } \Omega,\\
                                w
                                        =
                                                0,
				&
                                                \quad \text{on } \pa\Omega,
                \end{array}
        \right.
\end{equation}
where 
\begin{equation}\label{e:CH4_1_RHS_1111}
	\mathcal{R}_1 
		=
                \ep 
                \bigg(
			\Delta u^0
              + \Delta \eta
              + \dfrac{\partial^2 {\theta}_B}{\partial x^2}
              + \dfrac{\partial^2 {\theta}_L}{\partial y^2}
              + \Delta \eta 
              \bigg), 
\end{equation}
\begin{equation}\label{e:CH4_1_RHS_1112}
\begin{array}{rl}
	\mathcal{R}_2 
		=
	 &  
        \spacer
            \big(
                b_1(x, y) - b_1(0, y)
            \big)
            \dfrac{\partial {\theta}_L}{\partial x}
              + 
            \big(
                b_2(x, y) - b_2(x, 0)
            \big)
            \dfrac{\partial {\theta}_B}{\partial y}\\
        &
            +
            \big(
                b_1(x, y) - b_1(0, 0)
            \big)
            \dfrac{\partial \eta}{\partial x}
              + 
            \big(
                b_2(x, y) - b_2(0, 0)
            \big)
            \dfrac{\partial \eta}{\partial y}.
\end{array}
\end{equation}

Now, performing the energy estimates on the equation (\ref{e:w_eq_zero}) and 
using the estimates on the correctors, 
we state and prove the following convergence result.

\begin{theorem}[{\bf Non-characteristic boundary case}]\label{non_char_thm}
Assuming the data $f$ and $\blds{b}$ are sufficiently regular, 
the solution $u$ to (\ref{con_dif}) with $b_1, b_2 > 0$ (\textit{non-characteristic boundary case})  
satisfies the asymptotic expansion in (\ref{e:asymp_exp_non}) 
in the sense that  
\begin{align}\label{e:validity_non}
\|
    u - u^0 - {\theta}_B - {\theta}_L - {\eta}
\|_\ep 
        \leq 
            \kap \ep^{\frac{3}{4}}.
\end{align}
Moreover, 
$u$ converges to $u^0$ as the perturbation parameter $\ep$ tends to zero as
\begin{align}\label{e:VVL_non}
\|
    u - u^0 
\|_L^2(\Omega) 
        \leq 
            \kap \ep^{\frac{1}{4}}.
\end{align}
\end{theorem}
\begin{proof}
Multiply (\ref{e:w_eq_zero}) by $e^x w$ and integrating over $\Omega$, 
we write 
\begin{align}
\begin{split}
&\ep\|\nabla w \|^2_{L^2(\Omega)} + \frac{1-\ep}{2}\|w \|^2_{L^2(\Omega)} \\
&
    \qquad
        \leq
        \kappa
                        \left|\int_{\Omega}
                        \left[
                        \mathcal{R}_1
                        +
                        \mathcal{R}_2
                        + e.s.t.
                        \right]
                        w \, dxdy\right|\\
&
    \qquad
        \leq 
        \kappa \| \mathcal{R}_1 \|^2_{L^2(\Omega)} 
        + 
        \kappa \| \mathcal{R}_1 \|^2_{L^2(\Omega)}
        +
        \frac{1}{4}\|w \|^2_{L^2(\Omega)}. 
\end{split}
\label{CH4_2_FFF_0_1}
\end{align}
Using the explicit forms of the correctors and using (\ref{e:CH4_1_RHS_1111}) and (\ref{e:CH4_1_RHS_1112}), 
one can verify that 
\begin{equation}\label{est_rem_non}
        \| \mathcal{R}_1 \|_{L^2(\Omega)} 
            \leq \kap \ep, 
        \qquad 
        \| \mathcal{R}_1 \|_{L^2(\Omega)} 
            \leq \kap \ep^{\frac{3}{4}}.
\end{equation}
Combining (\ref{CH4_2_FFF_0_1}) and (\ref{est_rem_non}), 
we find that 
\begin{equation}
\ep\|\nabla w \|^2_{L^2(\Omega)} 
+ 
\frac{1}{8}\|w \|^2_{L^2(\Omega)} 
        \leq
        \kappa \ep^\frac{3}{2},  
\end{equation}
and hence we obtain (\ref{e:validity_non}).  
Thanks to the smallness of the correctors, (\ref{e:VVL_non}) follows from (\ref{e:validity_non}), 
and now the proof is complete. 
\end{proof}

\subsection{Characteristic boundary case ($b_1 > 0, b_2=0$)}\label{S.char}
When    
$b_1>0, \, b_2 = 0$, 
the flow associated with (\ref{con_dif}) 
is moving in the $x$-direction across 
the boundary at $x= 1$.  
Hence, 
by setting 
$\ep=0$ in (\ref{con_dif}) 
and imposing the {\it inflow boundary condition} at $x= 1$, 
we write the equation for the limit solution $u^0$ as  
\begin{align}\label{con_dif_limit_char}
        \left\{
                \begin{array}{rl}
                	\spacer        - b_1 \dfrac{\partial u^0}{\partial x} = f, &\ \text{in $\Omega$},\\
        u = 0, &\ \text{at $x=1$}.
                \end{array}
        \right.
\end{align}
Since $b_1 > 0$, we find the limit solution as  
\begin{align}\label{limit_sol_compati}
u^0 = \int^1_x \dfrac{f}{b_1}(x_1,y) \, dx_1.
\end{align}
The data ${b}_1$ and $f$ are assumed to be smooth enough and 
hence 
we have 
the limit solution $u^0$ sufficiently regular for the boundary layer analysis below.

To investigate the asymptotic behavior of $u^\ep$ (solution of (\ref{con_dif})) as $\ep \rightarrow 0$, 
we first notice that $ u - u^0 = 0$ on $x = 1$ only, 
but 
\begin{equation}\label{e:nonzero_boundary_chatacteristic}
    u-u^0 
    \neq 0 
    \text{ along the boundaries at $x = 0$, $x = 1$ and $y = 0$.}
\end{equation}
Two different types of boundary layers 
occur for this interesting problem, that is, 
parabolic boundary layers near the top and bottom boundaries at $y = 0$ and $y = 1$, 
and the ordinary boundary layer 
near the (left) out-flow boundary at $x = 0$. 
Moreover, the parabolic boundary layers and the ordinary layer interact near the corners at $(0, 0)$ and $(0, 1)$. 
The boundary layer analysis for this characteristic boundary case is in fact very technical and complex, 
and it is fully analyzed in the earlier works  \cite{Gri85, SK87, GJT13, GJT_Review, Book16}. 

Our main goal in this article is to construct a novel neural network, highly efficient for the singularly perturbed problem (\ref{con_dif}), 
by using the analytic information obtained from the associated boundary layer analysis. 
Hence we briefly recall below, especially from Section 4.1.1 in \cite{Book16}, 
the asymptotic expansion of $u$ 
as well as 
the convergence results. 
The correctors introduced in this section will play an important role in Section \ref{char} below 
when we construct our novel neural network for the singularly perturbed problem (\ref{con_dif}).

We construct the asymptotic expansion of $u$ at a small $\ep$ in the form, 
\begin{equation}\label{e:exp_char}
    u 
    \simeq  
        u^0 
        +
        \varphi_B + \varphi_T
        +
        \tht_L
        +
        \zeta_B + \zeta_T, 
\end{equation}
where, see Fig. \ref{f:correctors} below, 
\begin{itemize}
    \item [$u^0$:]
        corresponding limit solution defined in (\ref{limit_sol_compati}),  
    \item [${\varphi}_B$:]
        parabolic boundary layer corrector (PBL) near the bottom  boundary at $y = 0$,
    \item [${\varphi}_T$:]
        parabolic boundary layer corrector (PBL) near the top boundary at $y = 1$,  
    \item [${\tht}_L$:]
        ordinary boundary layer corrector (OBL) near the left boundary at $x = 0$,
    \item [$\zeta_B$ :]
        corner boundary layer corrector (CBL) that 
        manages the interaction between 
        ${\varphi}_B$ and ${\tht}_L$ near the corner at $(0, 0)$, 
    \item [$\zeta_T$ :]
        corner boundary layer corrector (CBL) that 
        manages the interaction between 
        ${\varphi}_T$ and ${\tht}_L$ near the corner at $(0, 1)$.
\end{itemize}
\begin{figure}[h] 
\begin{center}
\scalebox{.70} 
{
\begin{pspicture}(0,-3.16)(8.716875,3.16)
\psline[linewidth=0.04cm,arrowsize=0.05291667cm
2.0,arrowlength=1.4,arrowinset=0.4]{->}(0.056875,-1.86)(8.696875,-1.86)
\psline[linewidth=0.04cm,arrowsize=0.05291667cm
2.0,arrowlength=1.4,arrowinset=0.4]{->}(0.696875,-2.48)(0.696875,3.14)
\psframe[linewidth=0.04,dimen=outer](7.496875,2.34)(0.696875,-1.88)
\usefont{T1}{ptm}{m}{n}
\rput(4.2042187,0.225){\large ${u}^0$}
\usefont{T1}{ptm}{m}{n}
\rput(4.2042187,1.855){\large ${\varphi}_T$}
\usefont{T1}{ptm}{m}{n}
\rput(4.2042187,-1.45){\large ${\varphi}_B$}
\usefont{T1}{ptm}{m}{n}
\rput(1.2,1.855){\large ${\zeta}_T$}
\usefont{T1}{ptm}{m}{n}
\rput(1.2,-1.45){\large ${\zeta}_B$}
\usefont{T1}{ptm}{m}{n}
\rput(1.2,0.225){\large ${\tht}_L$}
\usefont{T1}{ptm}{m}{n}
\rput(0.3,-2.2){\large $0$}
\usefont{T1}{ptm}{m}{n}
\rput(0.3, 2.3){\large $1$}
\usefont{T1}{ptm}{m}{n}
\rput(0.3,2.9){\large $y$}
\usefont{T1}{ptm}{m}{n}
\rput(7.555,-2.2){\large $1$}
\usefont{T1}{ptm}{m}{n}
\rput(8.65,-2.2){\large $x$}
\end{pspicture}
}\caption{Limit solution, and parabolic, ordinary, and corner boundary layer correctors}\label{f:correctors}
\end{center}
\end{figure}
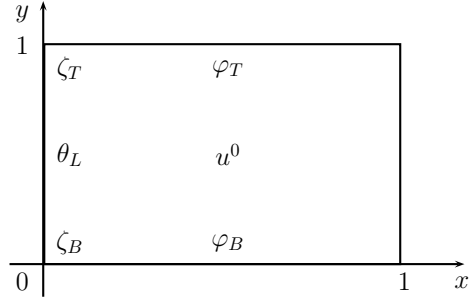

Recalling from Section 4.1.1 in \cite{Book16} that the correctors
$\varphi_B$, $\tht_L$, and $\zeta_B$ are defined as 
\begin{align}
         {{\varphi}}_B
                &
                =
                \sig_B(y) 
                {\ov{\varphi}}_B,
                \quad
                \text{with }
                %
                \label{pbl_zero_compati}
\end{align}
\begin{align}
         {\ov{\varphi}}_B
                &
                =
                -
                \sqrt{\frac{2}{\pi}}
                \int^\infty_{(b_1(0, 0) \bar{y})/\sqrt{2(1-x)}} \,
                        \exp\Big(-\frac{y_1^2}{2}\Big) \,
                        u^0\Big(x+\frac{(b_1(0, 0) \bar{y})^2}{2y_1^2},0\Big)
                        dy_1,  
                        \label{pbl_zero_compati_app}
\end{align}
where $\bar{y} = {y}/{\sqrt{\ep}}$ is the stretched variable in the parabolic layer near the bottom boundary at $y = 0$, and 
\begin{align}
    {\theta}_L 
    =
    \sig_L(x) \ov{\theta}_L, 
    \qquad \qquad
        &
        \ov{\theta}_L
        = 
            - u^0(0,y)  \exp\left(- b_1 (0, y) \frac{x}{\ep}\right),\label{theta0_explicit_bar}\\
    {{\zeta}}_B 
    =
    \sig_L(x) \sig_B(y) {\ov{\zeta}}_B, 
    \qquad \qquad
        &
        {\ov{\zeta}}_B
        = 
            - {{\varphi}_B}\big(0, y \big) 
            \exp\left(- b_1 (0, y) \frac{x}{\ep}\right).\label{zeta_explicit_bar_zero}
\end{align}
By the symmetry, the corrector $\varphi_T$ is defined by 
$\varphi_B$ with $y$ replaced by $1-y$, 
and 
$\zeta_T$ by $\zeta_B$ with
$\varphi_B$ and  $y$ replaced by $\varphi_T$ and $1-y$. 

The correctors, $\varphi_B$, $\tht_L$, and $\zeta_B$ defined above,  satisfy the equations, 
\begin{align}\label{para_charac_eqn}
        \left\{
                \begin{array}{rl}
                	\spacer
		-\ep \dfrac{\partial^2 {\varphi_B}}{\partial y^2} 
        - b_1 (0, 0) \dfrac{\partial  \varphi_B}{\partial x} 
        &
        = 
                    e.s.t.,
                    \\
                \spacer
                 {\varphi_B} 
                 &
                 = - u^0(x,0),
                 \quad
                 \text{at $y=0$},\\
                 \spacer
                 {\varphi_B} 
                 &
                 = \varphi_B(0, y) \neq0,
                 \quad
                 \text{at $x=0$},\\
                 {\varphi_B} 
                 &
                 = 0,
                 \quad
                 \text{at $x=1$ or $y=1$},
                \end{array}
        \right.
\end{align}
\begin{align}\label{ord_charac_eqn}
        \left\{
                \begin{array}{rl}
                	\spacer
		-\ep \dfrac{\partial^2 { {\theta}_L}}{\partial x^2} 
        - b_1 (0, y) \dfrac{\partial  { \theta}_L}{\partial x} 
        &
                    = 
                    e.s.t.,
                    \\
                \spacer
                 {{\theta}_L} 
                 &
                 = - u^0(0,y),
                 \quad
                 \text{at $x=0$},\\
                 \spacer
                 {{\theta}_L} 
                 &
                 = - \sig_L(x)u^0(0,0) \exp\left(- b_1 (0, 0)x/\ep \right),
                 \quad
                 \text{at $y=0$},\\
                 {{\theta}_L} 
                 &
                 = 0,
                 \quad
                 \text{at $x=1$ or $y=1$},
                \end{array}
        \right.
\end{align}
and 
\begin{align}\label{corn_charac_eqn}
        \left\{
                \begin{array}{rl}
                	\spacer
		-\ep \dfrac{\partial^2 { {{\zeta}}_B}}{\partial x^2} 
        - b_1 (0, y) \dfrac{\partial  {{\zeta}}_B}{\partial x} 
        &
                    = 
                    e.s.t.,
                    \\
                \spacer
                 {{{\zeta}}_B} 
                 &
                 = 
                    -\varphi_B(0, y)
                 \quad
                 \text{at $x=0$},\\
                 \spacer
                 {{{\zeta}}_B} 
                 &
                 = 
                    - \tht_L(x) (x,0), 
                 \quad
                 \text{at $y=0$},\\
                 {{{\zeta}}_B} 
                 &
                 = 0,
                 \quad
                 \text{at $x=1$ or $y=1$}.
                \end{array}
        \right.
\end{align}
By the symmetry, 
the corrector $\varphi_T$ is satisfies the equation (\ref{para_charac_eqn}) of 
$\varphi_B$ with $y = 0$ replaced by $y = 1$, 
and 
$\zeta_T$ satisfies (\ref{corn_charac_eqn}) of $\zeta_B$ with
$\varphi_B$ and  $y = 0$ replaced by $\varphi_T$ and $y = 1$.

Thanks to our construction of the correctors, 
we finally notice that the difference between $u$ 
and the proposed expansion is 
balanced on the boundary, i.e.,  
\begin{equation}\label{e:boundary_zero_char}
        u  - (u^0 
        +
        \varphi_B + \varphi_T
        +
        \tht_L
        +
        \zeta_B + \zeta_T)
        			=
			0,
			\quad
			\text{on }
			\pa \Omega.
\end{equation}

In order to validate our asymptotic expansion of $u$, 
we introduce
the difference between $u$ 
and the proposed expansion,
\begin{equation}\label{e:CH4_2_w0_eq}
    w
		=
		u
		-
		(u^0 
        +
        \varphi_B + \varphi_T
        +
        \tht_L
        +
        \zeta_B + \zeta_T),
\end{equation}
and perform the energy estimates 
as we did in the previous section for the non-characteristic boundary case. 
As a result, 
we recall the following convergence result from Theorem 4.1 in  \cite{Book16}:

\begin{theorem}[{\bf Characteristic boundary case}]\label{char_thm}
Assuming the data $f$ and $\blds{b}$ are sufficiently regular, 
the solution $u$ to (\ref{con_dif}) with $b_1 > 0$ and $b_2 = 0$ (\textit{characteristic boundary case})  
satisfies the asymptotic expansion in (\ref{e:exp_char}) 
in the sense that  
\begin{align}\label{e:validity_char}
\|
    u - 
        (u^0 
        +
        \varphi_B + \varphi_T
        +
        \tht_L
        +
        \zeta_B + \zeta_T
        )
\|_\ep 
        \leq 
            \kap \ep^{\frac{3}{4}}.
\end{align}
Moreover, 
$u$ converges to $u^0$ as the perturbation parameter $\ep$ tends to zero as
\begin{align}\label{e:VVL_char}
\|
    u - u^0 
\|_{L^2(\Omega)} 
        \leq 
            \kap \ep^{\frac{1}{4}}.
\end{align}
\end{theorem}

\begin{rem}\label{r:improved_char}
\textnormal{
The convergence rate (\ref{e:validity_char}) can be improved by imposing certain compatibility conditions on the data $f$. 
More precisely, we have:  
\begin{itemize}
    \item[1.] 
            If the data $f$ satisfies $f(1,0)=f(1,1)=0$, we have
            $$ 
            \|u - (u^0 
                    +
                    \varphi_B + \varphi_T
                    +
                    \tht_L
                    +
                    \zeta_B + \zeta_T
            )\|_\ep 
                \leq \kap \ep. 
            $$
    \item[2.]
            In case when $f(x,0)=f(x,1)=0$, $0 < x < 1$, then 
            we infer from (\ref{limit_sol_compati}) that $u^0(x,0)=u^0(x,1)=0$, and hence 
            we do not need to include 
            ${\varphi}_B$, ${\varphi}_T$, ${\zeta}_B$, or ${\zeta}_T$ in our assymptotic expansion (\ref{e:exp_char}). That is, we have, for this case, 
            $$ 
            \|u - (u^0 
                    +
                    \tht_L
            )\|_\ep 
                \leq \kap \ep.  
            $$
    \item[3.]
            All the estimates above remain valid with the correctors $\varphi$, $\tht$, and $\zeta$ 
            replaced by their approximations without cut-off functions $\sig_L$ and $\sig_B$ 
            denoted by 
            $\ov{\varphi}$, $\ov{\tht}$, and $\ov{\zeta}$. 
\end{itemize}
}
\end{rem}

\section{Numerical experiments}\label{sec_numerics}
For all the numerical experiments performed in this article below, we use the following setting, if not specified:
\begin{enumerate}
\item Number of epochs: EPOCHS = 1000.
\item Number of neurons: n = 32.
\item Loss function is defined by using the mean squared error, such as MSELoss in PyTorch.
\item Number of collocation points to evaluate the loss function: $N\times N$ with $N = 50$. That is, we choose $(x_i,y_j)\in\bar{\Omega}$, $i,j=1,2,\cdots,50$.
\end{enumerate}

Our main goal in this section is to construct semi-analytic 
(two-layer) Physics Informed Neural Networks (PINNs) to approximate the solution $u$ to (\ref{con_dif}) especially when the diffusion parameter $\ep>0$ is very small. 
It's important to note that our approach utilizes the two-layer neural network. This is not only computationally cost-effective but also sufficient for obtaining accurate numerical approximations. Moreover, our methodology can be conveniently expanded to an $M$-layer neural network by employing a symbolic computation.
The detailed construction of these new semi-analytic PINNs, 
hereafter we call {\it Singular Layer PINNs} (SL-PINNs), appear below for both {\it characteristic} boundary and {\it non-characteristic} boundary cases in Sections \ref{non-char} and \ref{char}, but we emphasize here that the main idea of our SL-PINNs is embedding the corrector (which describe the stiff part of the solution $u$) in the structure of the solution $u$ and let the neural network learn the remaining smooth part of the solution by incorporating with the corrector. 
To construct our SL-PINNs, which will be verified below as highly efficient networks for (\ref{con_dif}) (or any other singular perturbation problems), we first recall the 
following two-layer (usual) PINNs: 

Enforcing the zero boundary conditions to the approximation, 
we approximate $u$, sol. of (\ref{con_dif}), 
by the usual PINN approximation $\widetilde{u}$:
\begin{subequations}\label{nn}
\begin{align}
\widetilde{u}(x,y; \Tht)= x(1-x)y(1-y)\hat{u}(x, y ; \Tht), 
\qquad
\textit{(PINN approx. of $u$)}, 
\label{nn_app}
\end{align}
where  
\begin{align}\label{nn_hat}
\hat{u}(x,y; \Tht) 
    = \sum_{j=1}^n c_j\sigma(w_{1j} x + w_{2j} y + b_j), 
\quad
    \text{$n=$ number of neurons}.  
\end{align}
The network parameters are denoted by 
$$
\Tht = (w_{11}, \cdots, w_{1n}, 
        w_{21}, \cdots, w_{2n}, 
        b_{1 }, \cdots, b_{ n}, 
        c_{1 }, \cdots, c_{n}
        ), 
$$
and 
we choose the logistic sigmoid as an activation function,
\begin{equation}\label{e:sigmoid}
    \sigma(s) 
        =
            1/(1 + e^{-s}).
\end{equation}
\end{subequations}

As it is well-known to the community of PINNs and machine learning and as it is 
well-investigated in \cite{GHJ} for some 1D singular perturbation problems, 
when the singular perturbation problem (\ref{con_dif}) exhibits a stiff boundary layers at a small diffusivity $\ep$, 
the PINN method introduced above in (\ref{nn}), 
(or any other version of PINNs without any special treatment near the boundary layers), fails to 
an approximate solution to the problem. 
In fact, the usual PINNs like (\ref{nn}) (with multiple layers and more neurons) do not produce a  
reasonable approximation to the solution of any example in the following sections when the diffusivity parameter $\ep$ is relatively smaller than the mesh size determined by the number of collocation points.

We demonstrate below how we 
modify the usual PINNs, $\widetilde{u}$, in (\ref{nn}) and construct our new semi-analytic SL-PINNs for the singularly perturbed problem (\ref{con_dif}). 
That is, we enrich the PINNs with 
various types of boundary layer correctors as discussed in Sections \ref{non-char} and \ref{char}, depending on the nature of the boundaries, i.e., whether they are characteristic or non-characteristic.

We start with the {\bf \em non-characteristic boundary case}, for which 
the corresponding corrector functions are given in a simpler form compared to the case of characteristic boundary case.

\subsection{Non-characteristic boundary case ($b_1 ,b_2 > 0$)}\label{non-char}
 
When $b_1 ,b_2 > 0$, i.e., the boundaries are non-characteristic, as demonstrated in Section \ref{sec_non_char} and supported by Theorem \ref{non_char_thm}, 
the corrector functions for the boundary layers associated with (\ref{con_dif}) take the form of exponential functions (OBLs, CBLs) 
and hence they can be readily incorporated into the PINNs framework. 
The SL-PINNs  for this non-characteristic boundary case are  proposed in the following sections.

\subsubsection{Experiment 1: constant coefficients $b_1=b_2=1$}
For non-characteristic boundaries, in the experiment, We consider the problem (\ref{con_dif})) with 
$b_1=b_2=1$:
\begin{equation}\label{eq_constant_coeff}
    \left\{
    \begin{array}{rl}
        \spacer
        - \ep\Delta u 
        - \dfrac{\pa u}{\pa x} 
        - \dfrac{\pa u}{\pa y} = f, & \text{in $\Omega$}, \\
        u = 0, & \text{on $\partial\Omega$}.
    \end{array}
    \right.
\end{equation}
As indicated in Theorem \ref{non_char_thm}, to absorb the sharpness
of boundary layers, we write
\begin{subequations}\label{sl-pinn}
\begin{align}
u &=  h(x,y) -  h(0,y)e^{-x/\ep} + \mathcal{O}(\ep),
\end{align} 
where
\begin{align}
h(x,y) = u^0(x,y) - u^0(x,0)e^{-y/\ep}.
\end{align}
\end{subequations}
To approximate the solutions $u$ to (\ref{eq_constant_coeff}), 
dropping the $\mathcal{O}(\ep)$ and replacing $u^0$ by the neural net $\hat{u}$ in (\ref{sl-pinn}),
we propose 
the SL-PINN, 
\begin{subequations}\label{non_charac_model}
\begin{align}
    \widetilde{u}(x,y) 
        &=
            (x-1) \, 
            ( 
                A({ {x,y}})
            -
                A({ {0,y}}) \,
                e^{-x/\ep}
            ), \\
       A(x,y) 
        &=
            (y-1) \, 
            ( 
                \hat{u}({ {x,y}})
            -
                \hat{u}({ {x,0}}) \,
                e^{-y/\ep}
            ).          
\end{align}
\end{subequations}
Here, we impose the zero boundary condition to $\widetilde{u}$ 
at the inflows, $x=1$ or $y=1$. 

When we compute the loss function of SL-PINNs,  
using the fact that 
\begin{equation}\label{e:derivative_sigmoid}
\begin{array}{rlcrl}
    \sig^\pri(s)
        &
        \spacer
        =
            \sig(s) \big(
                        1 - \sig(s)
                    \big), 
    &&
    \sig^{\pri \pri}(s)
        &
        =
            \sig (s) 
                    \big(
                        1 - \sig(s)
                    \big)
                     \big(
                        1 - 2 \sig(s)
                    \big), 
\end{array}
\end{equation}
we use the explicit form of derivatives, e.g.,
\begin{equation}\label{e:NN_modified_der}
\begin{array}{rl}
\spacer
    \dfrac{\partial \hat{u}}{\partial x}
    \displaystyle 
    \spacer
    &
    \displaystyle
    =
        \sum_{j=1}^n
          c_j w_{1j}\sigma^{\pri}(w_{1j} x + w_{2j} y + b_j),\\
\dfrac{\partial^2 \hat{u}}{\partial x^2}
    \displaystyle 
    \spacer
    &
    \displaystyle
    =
        \sum_{j=1}^n
          c_j (w_{1j})^2\sigma^{\pri \pri}(w_{1j} x + w_{2j} y + b_j).
\end{array}
\end{equation}

For a first example, we investigate the singular behaviors of the solutions $u$ by computing the asymptotic errors and rates, which are shown in Table \ref{tb3} and Figure \ref{fig:plain4}. To do so, we choose $u^0 = (1-x)(1-y)$ as a limit solution to (\ref{eq_constant_coeff}) for $\ep=0$, and obtain the function $f=2-x-y$ accordingly. Specifically, we compute the $L^2$ and $L^\infty$ norms of the asymptotic error, $u - u^0 - \ov{\theta}_B - \ov{\theta}_L - \ov{\eta}$, i.e.,
\begin{align}\label{asym}
\widetilde{u} - u^0 + u^0(0,y)e^{-x/\ep} + u^0(x,0)e^{-y/\ep} - u^0(0,0)e^{-(x+y)/\ep},
\end{align}
where $\widetilde{u}$ is from SL-PINN (\ref{non_charac_model}) and $u^0 = (1-x)(1-y)$. Then, the estimates for the asymptotic error (\ref{asym}) are indicated in Theorem \ref{non_char_thm}. Figure \ref{fig:plain4} illustrates that the errors decrease as $\ep$ approaches zero.

 \begin{table}[H]
{\small
\begin{tabular}{|l|l|l|l|l|l|l|}
\hline
  & $N^2$
 &  $\ep = 10^{-1}$ & $\ep = 10^{-2}$  & $\ep = 10^{-3}$  & $\ep = 10^{-4}$   \\ \hline
$L^2(\Omega)$ & $100^2$ & $1.221 \times 10^{-1}$ & $2.125\times 10^{-3}$  & $2.247 \times 10^{-3}$ & $ 4.869 \times 10^{-3}$  \\ 
& $400^2$ & $1.529 \times 10^{-1}$ & $2.339 \times 10^{-3}$ & $1.095\times 10^{-3}$ & $8.316 \times 10^{-4}$ \\ \hline
$L^\infty(\Omega)$ & $100^2$ & $4.861 \times 10^{-1}$ & $2.413 \times 10^{-2}$  & $1.016 \times 10^{-2}$ & $ 1.138 \times 10^{-2}$  \\ 
& $400^2$ & $6.220 \times 10^{-1}$ & $2.408 \times 10^{-2}$ & $3.928 \times 10^{-3}$ & $2.6149 \times 10^{-3}$ \\ \hline
\end{tabular} \vspace{1mm}
\caption{$L^2$ and $L^\infty$ norms of the asymptotic error (\ref{asym}); $f=2-x-y$, EPOCHS $=  600$,  n = 32, where $N^2$ is the number of collocation points.} \label{tb3}
}

\end{table}


\begin{figure}[h!]
     \centering
     \begin{subfigure}[b]{0.65\textwidth}
         \centering
         \includegraphics[width=\textwidth]{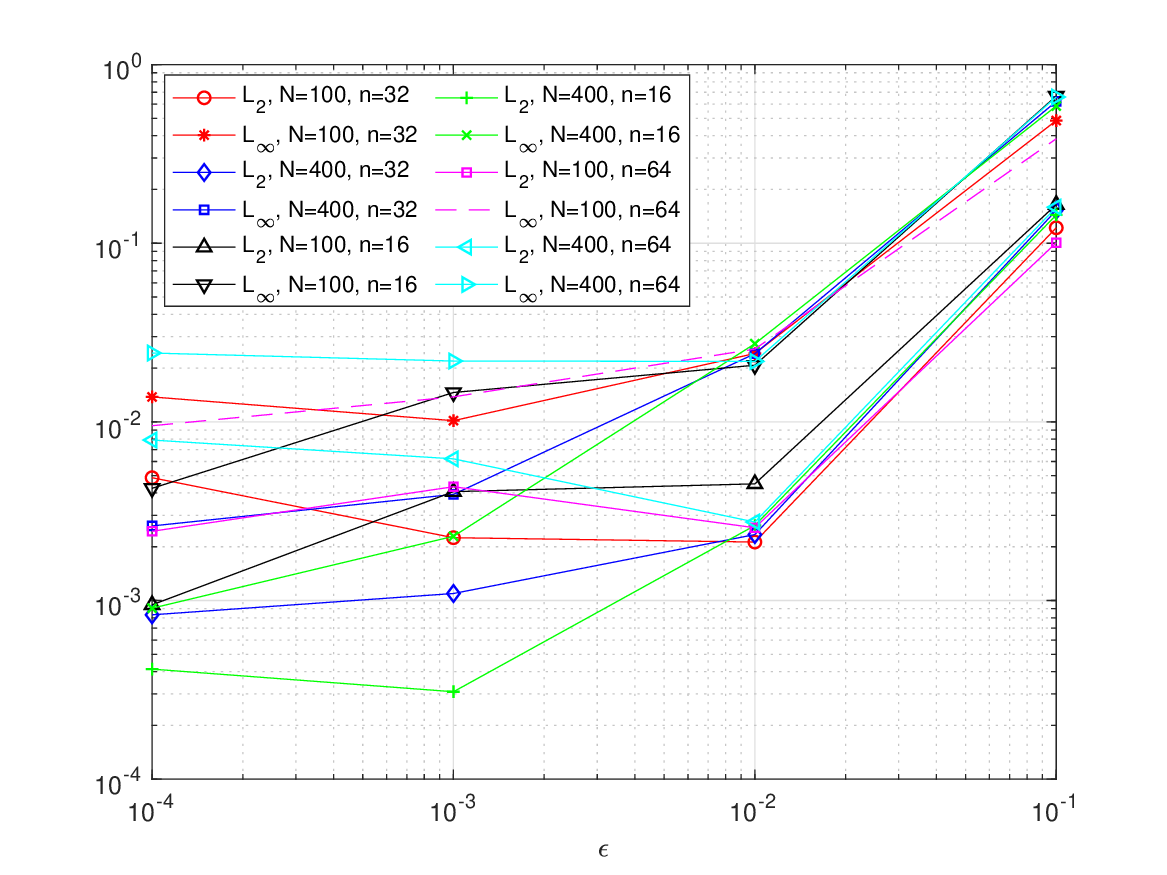}
     \end{subfigure}
     \caption{The $L^2(\Omega)$ and $L^\infty(\Omega)$ of (\ref{asym}) are presented with EPOCHS $=  600$, where $n$ and $N$ denote the number of neurons and collocation points, respectively. The numerical values for $n=32$ are shown in Table \ref{tb3}.}
        \label{fig:plain4}
\end{figure}

For another example, we take $f=1$. The surface plot of the corresponding solution is presented in (A), Figure \ref{fig:plain3} where the sharpness near the boundaries $x=0$, or $y=0$ are well captured.

\begin{figure}[h!]
     \centering
     \begin{subfigure}[b]{0.49\textwidth}
         \centering
         \includegraphics[width=\textwidth]{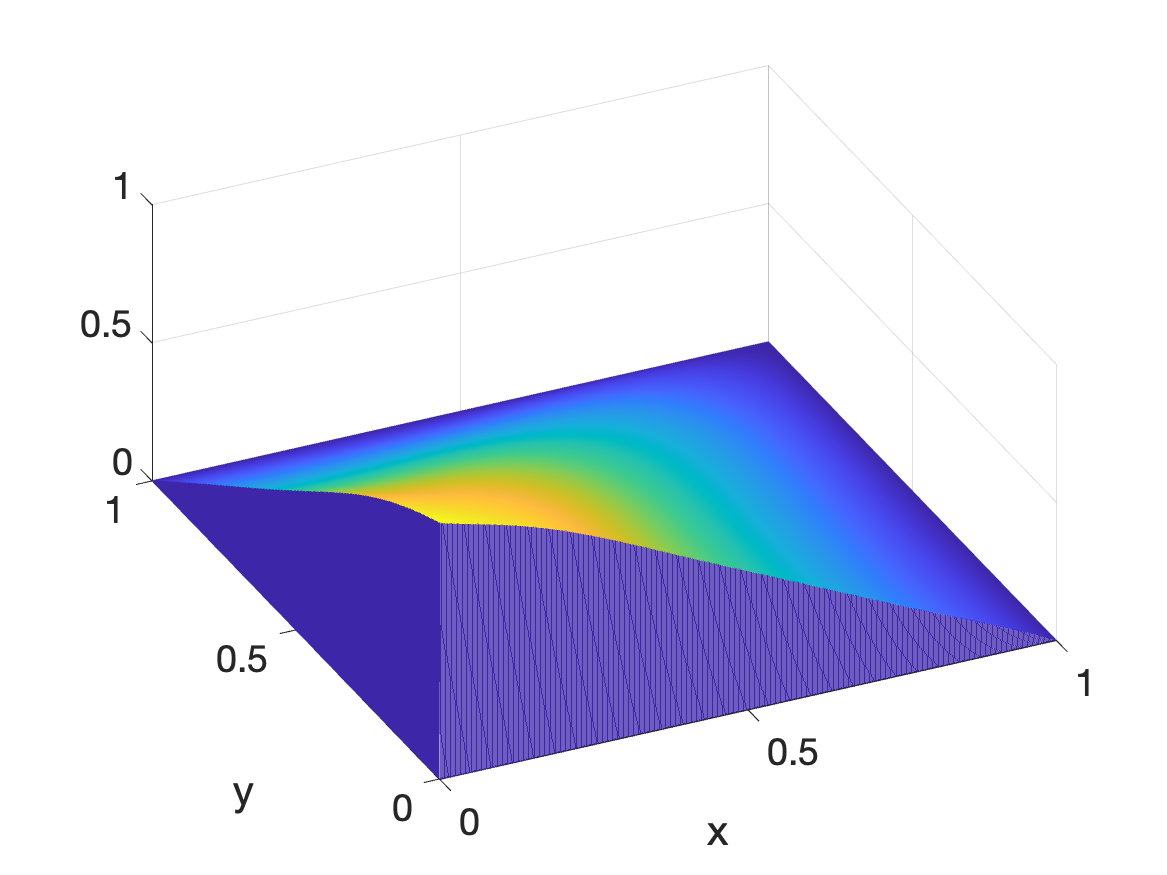}
         \caption{SL-PINN (\ref{non_charac_model}) for Problem (\ref{eq_constant_coeff}) \\ with $f=1$, $\ep= 10^{-5}$; EPOCHS $=  600$}
     \end{subfigure}
     \begin{subfigure}[b]{0.49\textwidth}
         \centering
         \includegraphics[width=\textwidth]{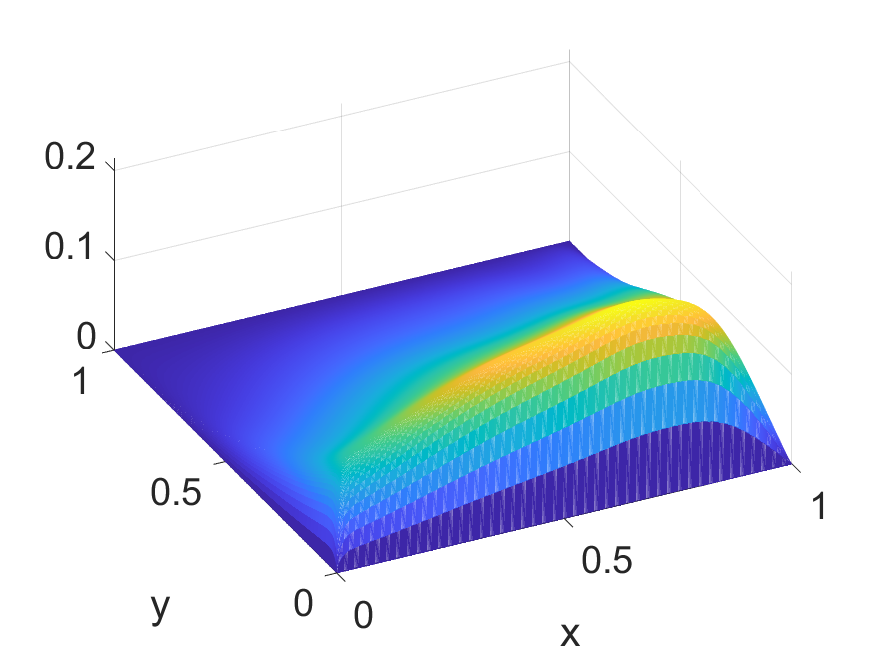}
         \caption{SL-PINN (\ref{cd_var_nn}) for Problem (\ref{cd_var}) \\ with $f=x$, $\ep= 10^{-2}$; EPOCHS $=  1000$}
     \end{subfigure}
     \caption{Surface plots of SL-PINN approximations $\widetilde{u}$.}
        \label{fig:plain3}
\end{figure}

\subsubsection{Experiment 2: variable coefficients $b_1,b_2>0$}
We consider  the convection-diffusion equation (\ref{con_dif}) with 
$b_1 =\cos(\frac{x}{2}+y)$,
$b_2 = \exp(x+y)$: 
\begin{equation}\label{cd_var}
    \left\{
    \begin{array}{rl}
        \spacer
        - \ep\Delta u 
        - \cos\Big(\dfrac{x}{2}+y \Big) 
            \dfrac{\pa u}{\pa x} 
        - \exp(x+y)
            \dfrac{\pa u}{\pa y} = f, & \text{in $\Omega$}, \\
        u = 0, & \text{on $\partial\Omega$}.
    \end{array}
    \right.
\end{equation}

We notice that 
$b_1 =\cos(x/2+y )>0$ and $b_2 = \exp(x+y)> 0$ in $\Omega$ and thus, 
following the boundary layer analysis in  Section \ref{sec_non_char}, 
we find that 
\begin{subequations}
\begin{align}\label{theta01}
     {\ov{\theta}_B} &= - u^0(x,0) \exp\left(- \exp (x)\frac{y}{\ep}\right),\\
     {\ov{\theta}_L} &= - u^0(0,y) \exp\left(- \cos(y)\frac{x}{\ep}\right),\\
     {\ov{\eta}} &= u^0(0,0)\exp\left(- \frac{\cos(y)x+\exp(x)y}{ \ep}\right).
\end{align}
\end{subequations}
As done in (\ref{sl-pinn}) and (\ref{non_charac_model}), 
we use the corrector functions above and propose the  SL-PINN as 
\begin{subequations}\label{cd_var_nn}
 \begin{align}
    \widetilde{u}(x,y ) 
        &=
            (x-1) \, 
            \Big(
                A({ {x,y}})
            -
                A({ {0,y}}) \,
                \exp(-\cos(y)x/\ep)
            \Big),  \\
       A(x,y ) 
        &=
            (y-1) \, 
            \Big( 
                \hat{u}({ {x,y}})
            -
                \hat{u}({ {x,0}}) \,
                \exp(-\exp(x)y/\ep)
            \Big).          
\end{align}
\end{subequations}

The surface plot of solution to (\ref{cd_var}) for $f=x$ is shown in (B), Figure \ref{fig:plain3}. The sharpness of the solutions near the boundaries are well captured.
\bigskip

Now we move to the more interesting 
{\bf \em characteristic boundary case}.

 \subsection{Characteristic boundary case ($b_1  > 0, b_2=0$)}\label{char}

As we observed in Section \ref{S.char}, 
the small parameter $\ep>0$ in the equation (\ref{con_dif}) results in the formation of parabolic boundary layers (PBLs, ${\varphi}_B$ and ${\varphi}_T$) along the characteristic boundaries at $y=0$ and $y=1$. 
In fact, it is more difficult, and on the other hand more interesting, 
to embedding these parabolic correctors in the structure of PINN approximations than the other 
simpler ordinary corrector and corner correctors. 
Hence, as a first step, we consider the case of compatible data $f$, satisfying  $f(x,0)=f(x,1)=0$, 
for which 
the parabolic boundary layers (nor the corner layers) do not occur; see Remark \ref{r:improved_char}. As a result, we only incorporate the correctors for OBLs ($\bar{\theta}$ of $\tht$) in the neural network approximations, as shown in Section \ref{S.char}   below. 
In the following section, Section \ref{S.exp4}, we also examine the case of non-compatible $f$, which may not satisfy $f(x,0)=0$ or $f(x,1)=0$. In this case, we anticipate the formation of the PBLs and corner boundary layers (CBLs, $\bar{\zeta}_B$ of $\zeta_B$ and $\bar{\zeta}_T$ of $zeta_T$) at the intersection of the OBLs and PBLs.

%

\subsubsection{Experiment 3: compatible $f$ s.t.  $f(x,0)=f(x,1)=0$}\label{S.exp3}
We consider the equation (\ref{con_dif}) with 
$b_1=1$, $b_2=0$ and 
$f(x,0)=f(x,1)=0$: 
\begin{equation}\label{compti_f}
    \left\{
    \begin{array}{rl}
        \spacer
        - \ep\Delta u 
        - \dfrac{\pa u}{\pa x} = f, & \text{in $\Omega$}, \\
        u = 0, & \text{on $\partial\Omega$}.
    \end{array}
    \right.
\end{equation}
Since $\bar{\varphi}_B = \bar{\varphi}_T = \bar{\zeta}_B = \bar{\zeta}_T = 0$, as indicated in Remark \ref{r:improved_char}, 
we write
\begin{align}
u &=  u^0(x,y) -  u^0(0,y)e^{-x/\ep} + \mathcal{O}(\ep).
\end{align} 
We thus approximate the solution $u$ to (\ref{compti_f})
by the SL-PINN,
\begin{align}\label{compati_model}
    \widetilde{u}(x,y) 
        &=
            (x-1)y(y-1) \, 
            ( 
                \hat{u}({ {x,y}})
            -
                \hat{u}({ {0,y}}) \,
                e^{-x/\ep}
            ).          
\end{align}

For the experiment, we take $f=\sin(\pi y)$ which is compatible at $y=0,1$, i.e., $f(x,0)=f(x,1)=0$. 
\begin{figure}[h!]
     \centering
     \begin{subfigure}[b]{0.49\textwidth}
         \centering
          \includegraphics[width=\textwidth]{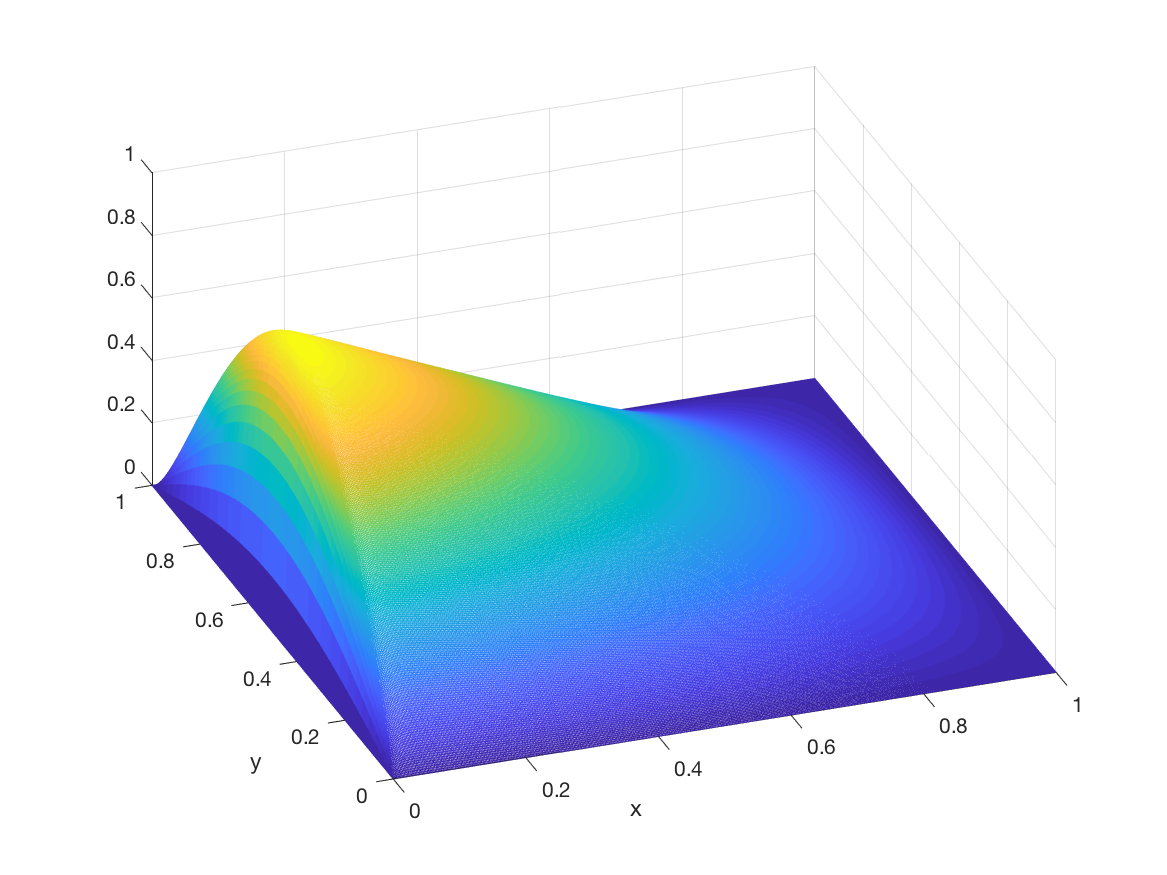}
         \caption{$\ep = 10^{-2}$}
         \label{comp-char-bdry}
     \end{subfigure}
     \begin{subfigure}[b]{0.49\textwidth}
         \centering
          \includegraphics[width=\textwidth]{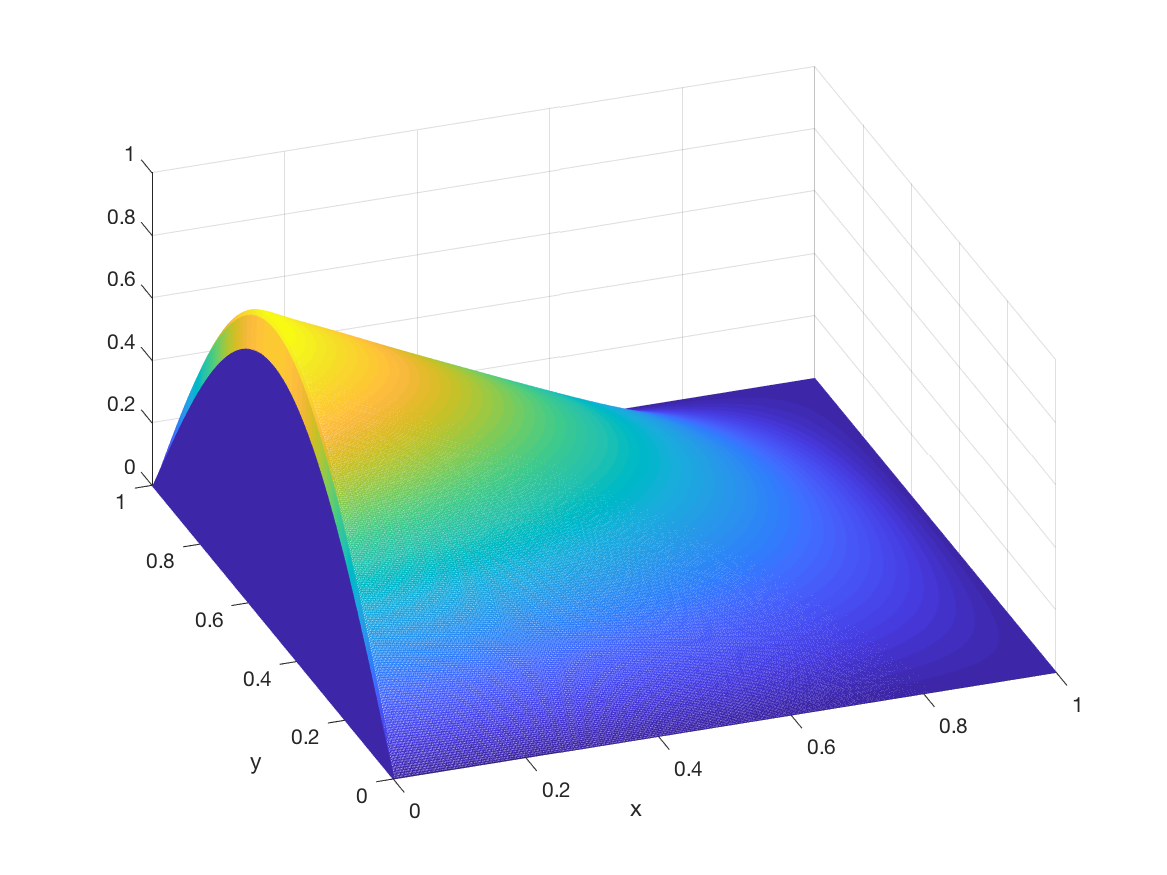}
         \caption{$\ep = 10^{-3}$}
         \label{comp-char-bdry}
     \end{subfigure}
     \caption{The surface plots of the neural net approximation $\widetilde{u}$ 
     by SL-PINN (\ref{compati_model}) for Problem (\ref{compti_f}) with $f=\sin(\pi y)$.}
        \label{fig:plain-1}
\end{figure}

For the comparison purpose, we use the exact solution to (\ref{compti_f}),
\begin{subequations}
\begin{equation}\label{true_sol}
    u(x,y) = (Ae^{r^{+}x} + Be^{r^{-}x} + \frac{1}{\ep\pi^2})\sin(\pi y) 
\end{equation}
where 
\begin{equation}
    A = \frac{e^{r^{-}} - 1}{\ep\pi^2 (e^{r^{+}} - e^{r^{-}})} ;\quad 
    B = \frac{1 - e^{r^{+}}}{\ep\pi^2 (e^{r^{+}} - e^{r^{-}})} ;\quad 
    r^{\pm} = \frac{1}{2\ep}(-1\pm \sqrt{1 + 4\ep^2\pi^2}).
\end{equation}   
\end{subequations}

We obtain the numerical errors $u - \widetilde{u}$ in $L^2(\Omega)$ and $L^\infty(\Omega)$ as shown in Table \ref{err_compati}.
We also plot them in Figure \ref{fig:plain1}.
 \begin{table}[H]
{\small
\begin{tabular}{|l|l|l|l|l|l|l|}
\hline
 & $N^2$
 &  $\ep = 10^{-1}$ & $\ep = 10^{-2}$  & $\ep = 10^{-3}$  & $\ep = 10^{-4}$  & $\ep = 10^{-5}$  \\ \hline
$\|u - \widetilde{u}\|_{L^2(\Omega)}$ &$100^2$ & $1.350 \times 10^{-4}$ & $2.404 \times 10^{-4}$ & $1.693 \times 10^{-3}$ & $1.019 \times 10^{-3}$ & $3.933 \times 10^{-3}$ \\ 
&$400^2$ & $1.538 \times 10^{-4}$ & $1.531 \times 10^{-4}$ & $7.894 \times 10^{-4}$ & $2.457 \times 10^{-3}$ & $3.566 \times 10^{-4}$ \\ \hline
$\|u - \widetilde{u}\|_{L^\infty(\Omega)}$ & $100^2$ & $3.552 \times 10^{-4}$ & $7.675\times 10^{-4}$ & $6.210 \times 10^{-3}$ & $3.355 \times 10^{-3}$ & $1.229 \times 10^{-2}$ \\
& $400^2$ & $3.688\times 10^{-4}$ & $4.002 \times 10^{-4}$ & $2.965 \times 10^{-3}$ & $7.644 \times 10^{-3}$ & $1.186 \times 10^{-3}$ \\ \hline
\end{tabular} \vspace{1mm}
\caption{$L^2$ and $L^\infty$ errors ($\|u - \widetilde{u}\|_{L^2(\Omega)}$ and $\|u - \widetilde{u}\|_{L^\infty(\Omega)}$) for SL-PINN (\ref{compati_model}) and Problem (\ref{compti_f}) with $f=\sin(\pi y)$.} \label{err_compati}
}
\end{table}


\begin{figure}[h!]
     \centering
     \begin{subfigure}[b]{0.49\textwidth}
         \centering
         \includegraphics[width=\textwidth]{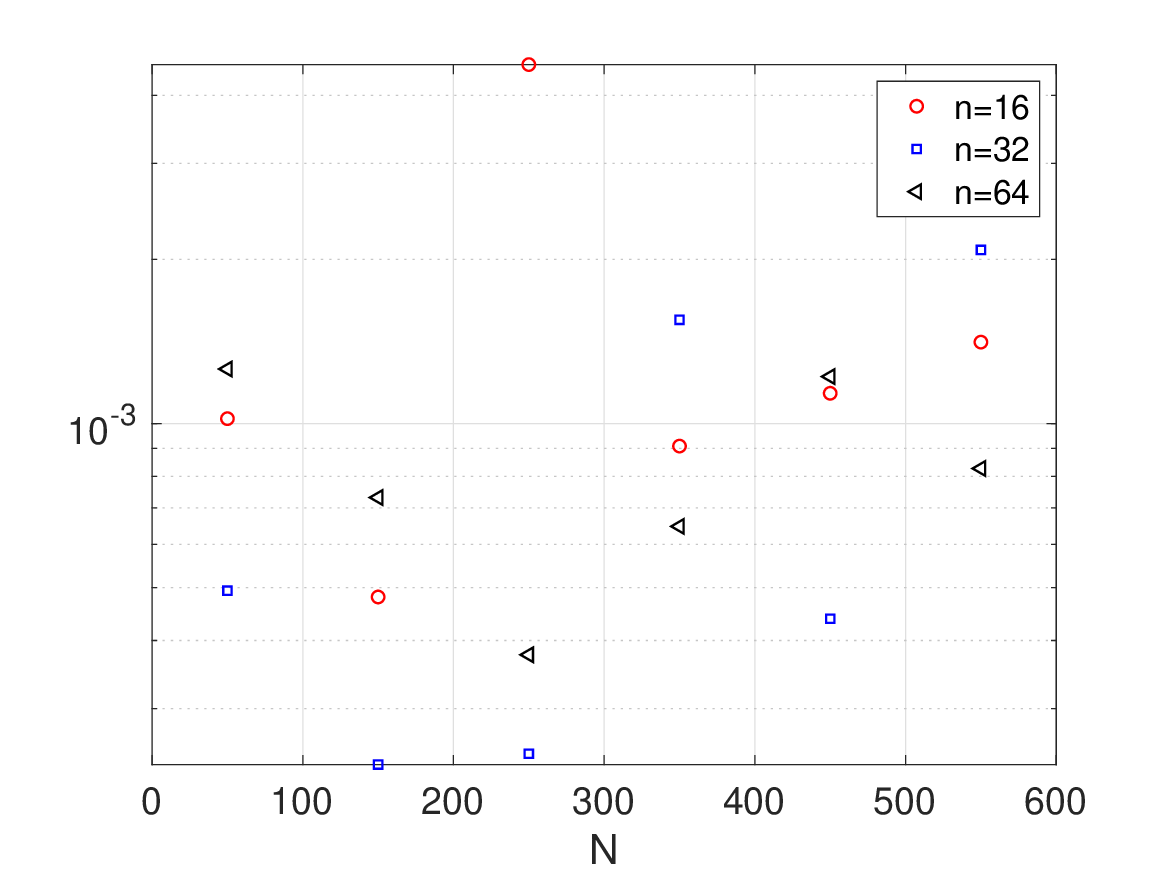}
         \caption{$\|u - \widetilde{u}\|_{L^2(\Omega)}$}
         \label{err_plot_compati}
     \end{subfigure}
     \begin{subfigure}[b]{0.49\textwidth}
         \centering
         \includegraphics[width=\textwidth]{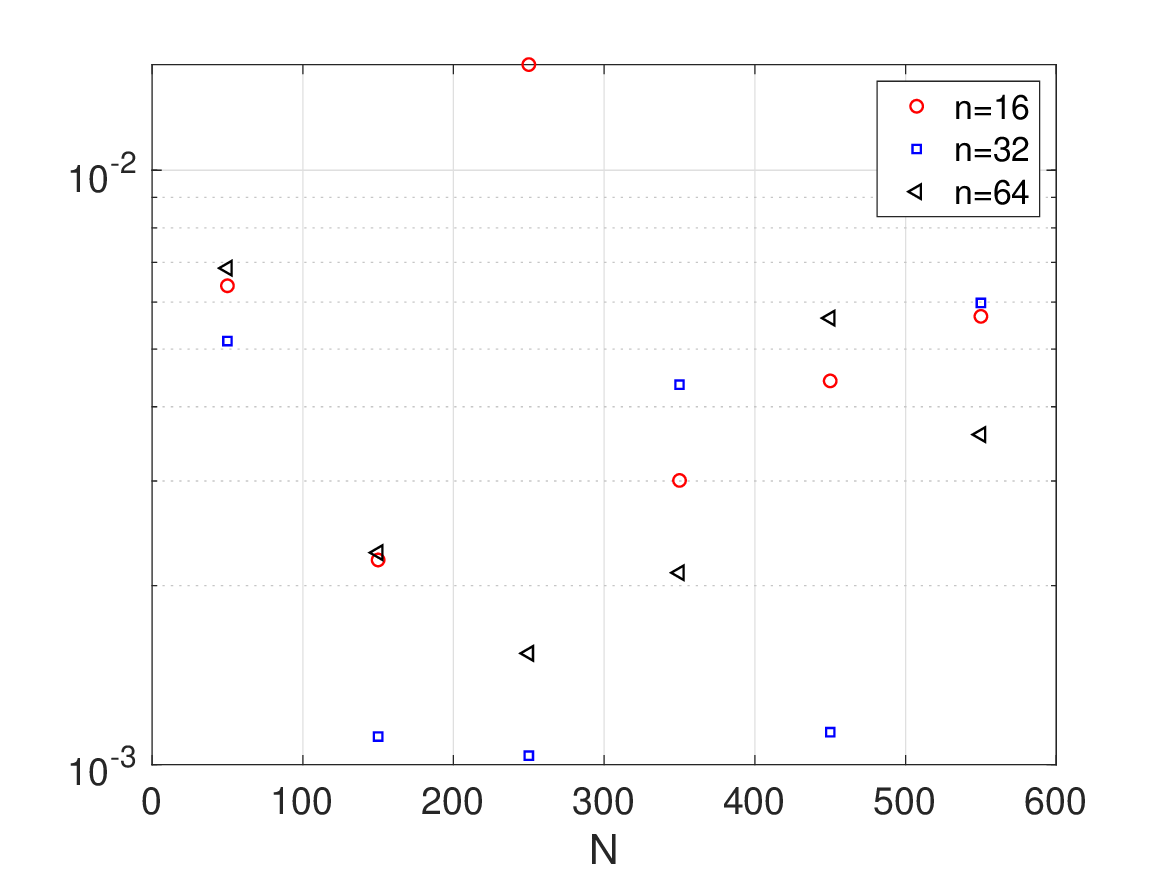}
         \caption{$\|u - \widetilde{u}\|_{L^\infty(\Omega)}$}
         \label{err_plot_compati}
     \end{subfigure}
     \caption{SL-PINN (\ref{compati_model}) for Problem (\ref{compti_f}) with $f=\sin(\pi y)$, $\ep = 10^{-3}$; the number of neurons = $n$ = 16, 32, 64; $N^2$ is the number of collocation points.}
        \label{fig:plain1}
\end{figure}

As we see in Figure \ref{fig:plain1}, it is evident that if the neural network size is too large ($n=64$), the network results in poor interpolation between the collocation points. Conversely, if the network size is too small ($n=16$), it fails to capture the sharpness of the solutions, leading to poor interpolation results. Optimal performance is achieved when $n=32$, resulting in lower $L^2$ and $L^\infty$ errors compared to the other network sizes.


%

\begin{figure}[h!]
     \centering
     \begin{subfigure}[b]{0.49\textwidth}
         \centering
         \includegraphics[width=\textwidth]{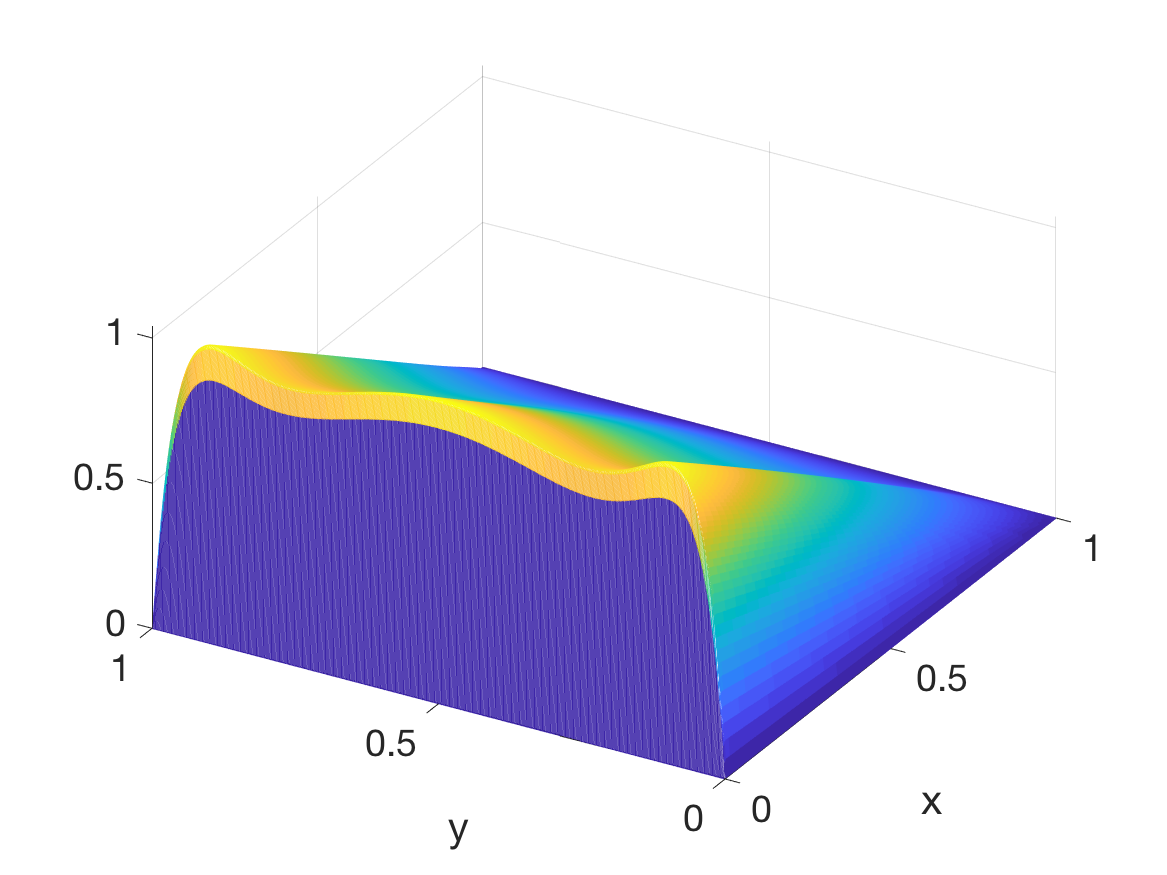}
         \caption{$\widetilde{u}$  in (\ref{compati_model}); EPOCHS=1000}
     \end{subfigure}
     \begin{subfigure}[b]{0.49\textwidth}
         \centering
         \includegraphics[width=\textwidth]{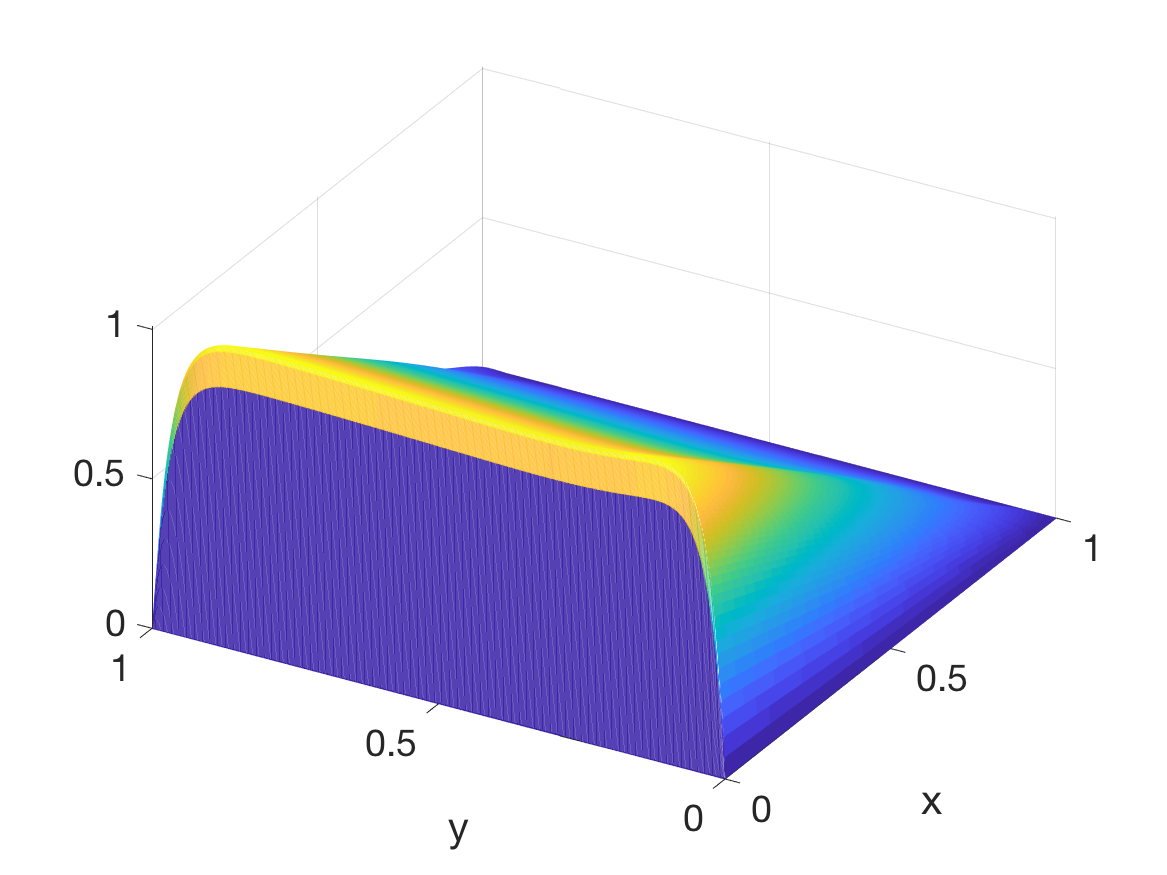}
         \caption{$\widetilde{u}$  in (\ref{compati_model_arti_pbl}); EPOCHS=3000}
         \label{non-comp-char-bdry2}
     \end{subfigure}
     \begin{subfigure}[b]{0.49\textwidth}
         \centering
         \includegraphics[width=\textwidth]{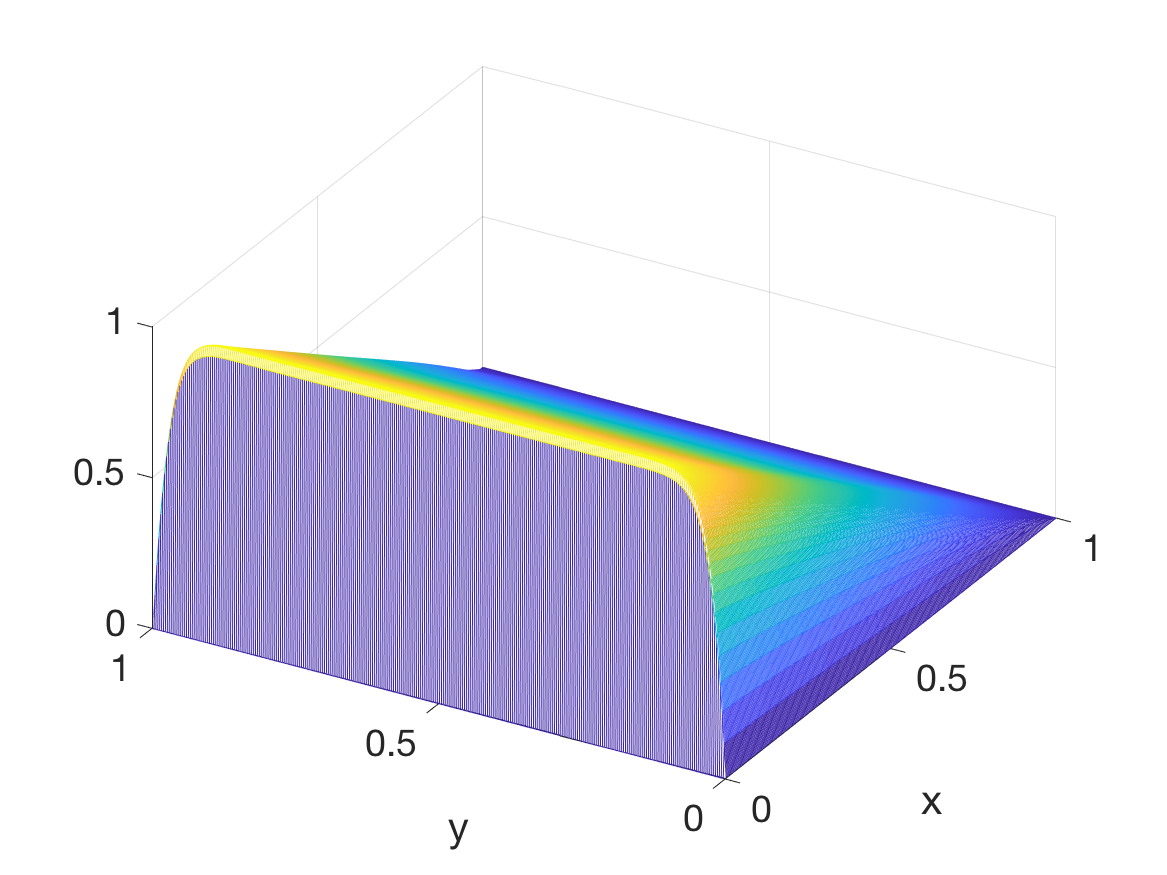}
         \caption{$\widetilde{u}$ in (\ref{pbl_sys_SL-PINN}); EPOCHS=3500}
     \end{subfigure}
     \begin{subfigure}[b]{0.49\textwidth}
         \centering
         \includegraphics[width=\textwidth]{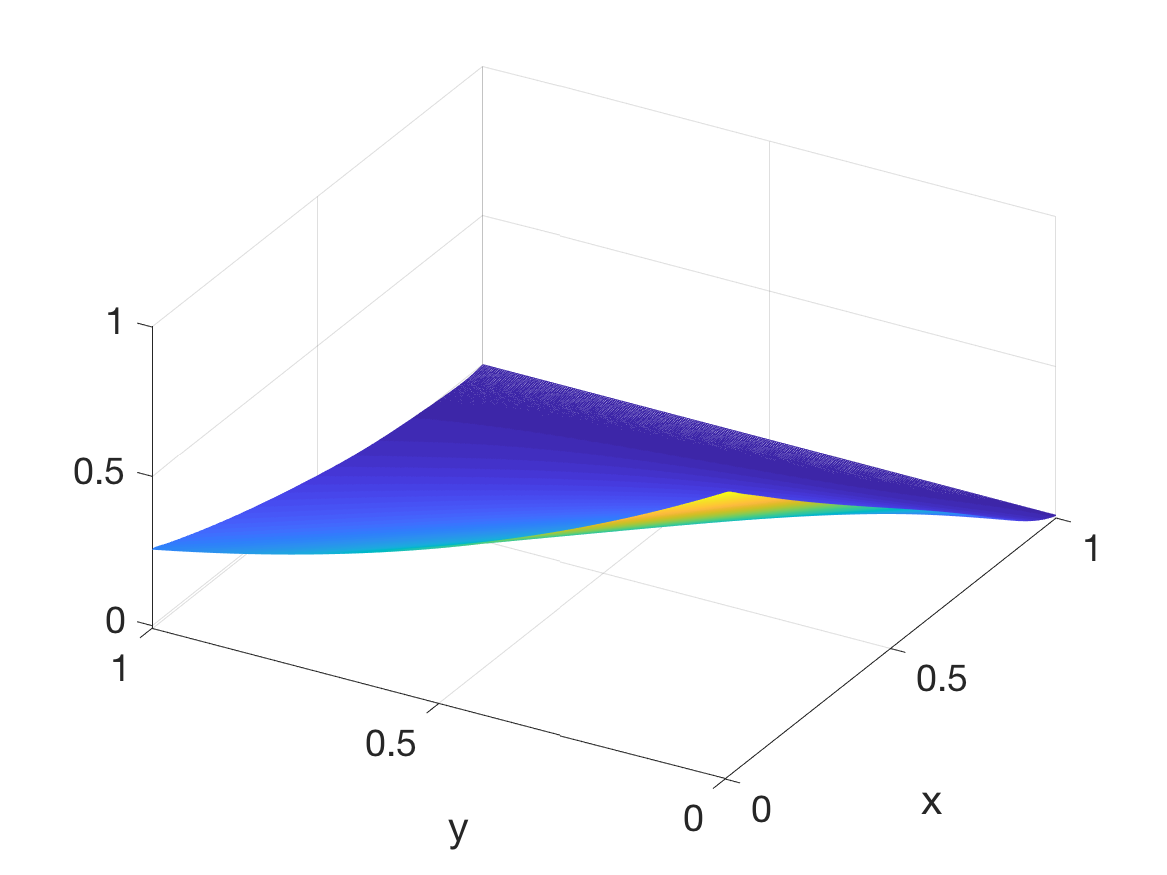}
         \caption{$\widetilde{\varphi}_{B}$ in (\ref{pbl_sys_SL-PINN}), the PBL at $y=0$ for (C)}
     \end{subfigure}     
     \caption{The surface plots of the neural net approximation $\widetilde{u}$  for Problem (\ref{compti_f}) with $f=1, \ep= 10^{-3}$; The number of EPOCHS is chosen to ensure convergence in stochastic optimization.}
        \label{fig:plain3-11}
\end{figure}


\subsubsection{Experiment 4: non-compatible $f$, i.e., $f(x,0)\neq 0$ or $f(x,1)\neq 0$}\label{S.exp4}
If we use the SL-PINN (\ref{compati_model}) for the non-compatible $f$, as in Figure \ref{fig:plain3-11} - (A), 
we observe the large over-shootings near the edges $y=0,1$. As indicated in Theorem \ref{char_thm}, it is due to the parabolic boundary layers (PBLs) and near the corners $(0,0)$, $(0,1)$ due to the corner layers (CBLs).

As our {\em first} and {\em simple} try to resolve the large computational error of 
the SL-PINN (\ref{compati_model}) 
near the edges $y=0,1$, 
we introduce the simple exponentially decaying functions, 
$e^{-y/\sqrt{\ep}}$ and  
$e^{-(1-y)/\sqrt{\ep}}$  near $y=0,1$, and 
construct the 
following SL-PINN enriched with those functions:
\begin{subequations}\label{compati_model_arti_pbl}
\begin{align}
    \widetilde{u}(x,y ) 
        &=
            (x-1) \, 
            ( 
                A({ {x,y}})
            -
                A({ {0,y}}) \,
                e^{-x/\ep}
            ), \\
       A(x,y ) 
        &=
                \hat{u}({ {x,y}}) - 
                 \hat{u}({ {x,0}})(1-y)e^{-y/\sqrt{\ep}}
            -
                 \hat{u}({ {x,1}})y e^{-(1-y)/\sqrt{\ep}},            
\end{align}
\end{subequations}
Compared to the PBLs $\bar{\varphi}_B$ in (\ref{pbl_zero_compati_app}), 
these empirically obtained exponential functions are preferred due to their simplicity of implementation using the exponential functions and their ability to effectively capture the sharpness of the layers. 
However, as seen in Figure \ref{fig:plain3-11} (B), we still observe small over-shootings near $(0,0)$ and $(0,1)$, indicating that the empirically obtained exponential functions do not fully capture the sharpness of the PBLs and CBLs.


In order to fully account for the small over-shootings and achieve a sharp approximation, 
we require the use of all the PBLs $\bar{\varphi}_B$, $\bar{\varphi}_T$, 
the CBLs $\bar{\zeta}_B$, $\bar{\zeta}_T$, and 
the OBL $\bar{\theta}_L$. As indicated in Remark \ref{r:improved_char} for non-compatible $f$, 
we  write
\begin{subequations}\label{sl-pinn2}
\begin{align}
u &=  h(x,y) -  h(0,y)e^{-x/\ep} + \mathcal{O}(\ep^{\frac{3}{4}}),
\end{align} 
where
\begin{align}
h(x,y) = u^0(x,y) - \bar{\varphi}_B(x,y) - \bar{\varphi}_T(x,y).
\end{align}
\end{subequations}

To construct the SL-PINN, 
we drop the $\mathcal{O}(\ep^{\frac{3}{4}})$,  
replace $u^0$ with the neural net approximation $\hat{u}$, 
and write 
the SL-PINN approximation $\widetilde{u}$ as 
\begin{subequations}\label{compati_model_pbl}
\begin{align}
    \widetilde{u}(x,y ) 
        &=  
        (x-1)
         \, 
            ( 
                A({ {x,y}})
            -
                A({ {0,y}}) \,
                e^{-x/\ep}
            ), \\
       A(x,y ) 
        &=
                \hat{u}({ {x,y}}) - \bar{\varphi}_B(x,y)
            -  \bar{\varphi}_T(x,y),           
\end{align}
\end{subequations}
where the heat solutions 
$\bar{\varphi}_B$ and $ \bar{\varphi}_T$ are defined in (\ref{pbl_zero_compati_app}), 
Because these parabolic correctors are involved in convolution integrals, 
implementing SL-PINN with those correctors is
not practical nor convenient.

As an alternative, we 
propose to construct the SL-PINN 
for the following system of equations for $u$,  $\ov{\varphi}_B$ and $\ov{\varphi}_T$ all together: 

\begin{subequations}\label{pbl_sys}
\begin{align}
- \ep\Delta u 
- \dfrac{\pa u}{\pa x} &= f,
\quad \text{in } \Omega,\\
\spacer
- \dfrac{ \pa^2 \bar{\varphi}_{B}}{\pa {y}^2} - 
\dfrac{\pa \bar{\varphi}_{B}}{\pa x} &= 0, 
\quad \text{in } \Omega,\\ 
- \dfrac{ \pa^2 \bar{\varphi}_{B}}{\pa {y}^2} - 
\dfrac{\pa \bar{\varphi}_{B}}{\pa x} &= 0, 
\quad \text{in } \Omega,\\
\end{align}
supplemented with the boundary conditions  
\begin{align}
& u = 0 \text{ on } \pa\Omega,\\ 
&\bar{\varphi}_{B}(x,0) = - {u}^0({ {x,0}}),\ \bar{\varphi}_{B}(1,y) = 0,\ \bar{\varphi}_{B}\rightarrow 0\ \text{as ${y}\rightarrow\infty$},\\
&\bar{\varphi}_{T}(x,1) =  - {u}^0({ {x,1}}),\ \bar{\varphi}_{T}(1,y) = 0,\ \bar{\varphi}_{T}\rightarrow 0\ \text{as ${y}\rightarrow -\infty$}.
\end{align}
\end{subequations}

To approximate $\bar{\varphi}_{B}, \bar{\varphi}_{T}$, we respectively use the neural net approximations $\widetilde{\varphi}_{B}, \widetilde{\varphi}_{T}$,
\begin{subequations}\label{pbl_sys_SL-PINN}
\begin{align}
\widetilde{\varphi}_{B} &= (x-1)\big(e^{-\bar{y}}\hat{u}({ {x,0}}) + (e^{-\bar{y}}-e^{-2\bar{y}})\hat{\varphi}_{B}    \big),\\
\widetilde{\varphi}_{T} &= (x-1)\big(e^{-\tilde{y}}\hat{u}({ {x,1}}) + (e^{-\tilde{y}}-e^{-2\tilde{y}})\hat{\varphi}_{T}\big),
\end{align}
and, to approximate the solution $u$, we use $\widetilde{u}$:
\begin{align}
    \widetilde{u}(x,y ) 
        &=
                A({ {x,y}})
            -
                A({ {0,y}})(1-x) \,
                e^{-x/\ep}, \\
       A(x,y ) 
        &=
                (x-1)\hat{u}({ {x,y}}) - 
                 (1-y)\widetilde{\varphi}_B
            -
                 y \widetilde{\varphi}_T.         
\end{align}
\end{subequations}

Figure \ref{fig:plain3-11}, (D) shows the neural net approximation $\widetilde{\varphi}_{B}$, while the $\widetilde{\varphi}_{T}$ can be shown similarly. By utilizing the corrector functions, $\widetilde{\varphi}_{T}$, $\widetilde{\varphi}_{T}$, as shown in (C), we can accurately approximate the solution $u$ by $\widetilde{u}$ in (\ref{pbl_sys_SL-PINN}). As opposed to (A) and (B), no over-shootings are observed at the corners $(0,0)$, $(0,1)$ or the edges $y=0,1$.


%

%
%


 \begin{figure}[h!]
     \centering
     \begin{subfigure}[b]{0.49\textwidth}
         \centering
         \includegraphics[width=\textwidth]{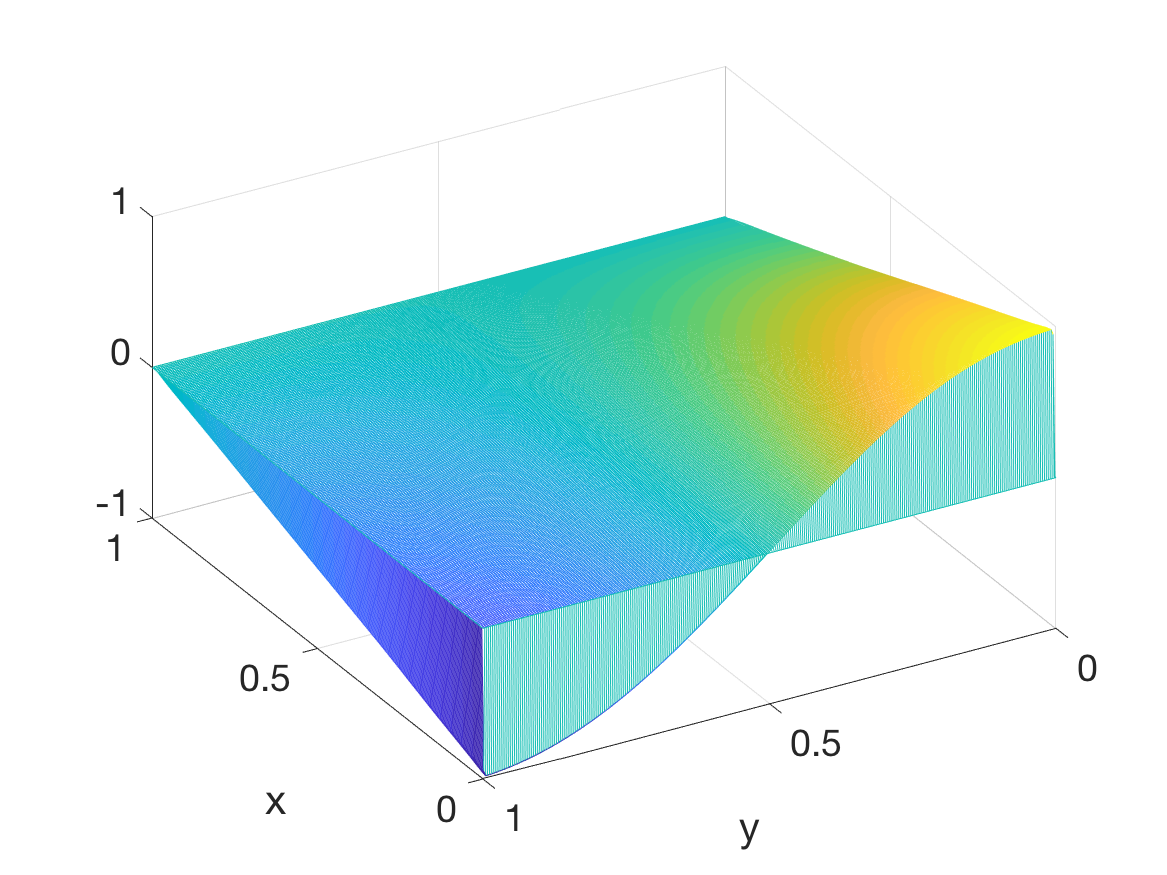}
         \caption{$f =\cos(\pi y)$, $\ep= 10^{-6}$; EPOCHS=3500, $N^2 = 50^2$}
     \end{subfigure}
     \begin{subfigure}[b]{0.49\textwidth}
         \centering
         \includegraphics[width=\textwidth]{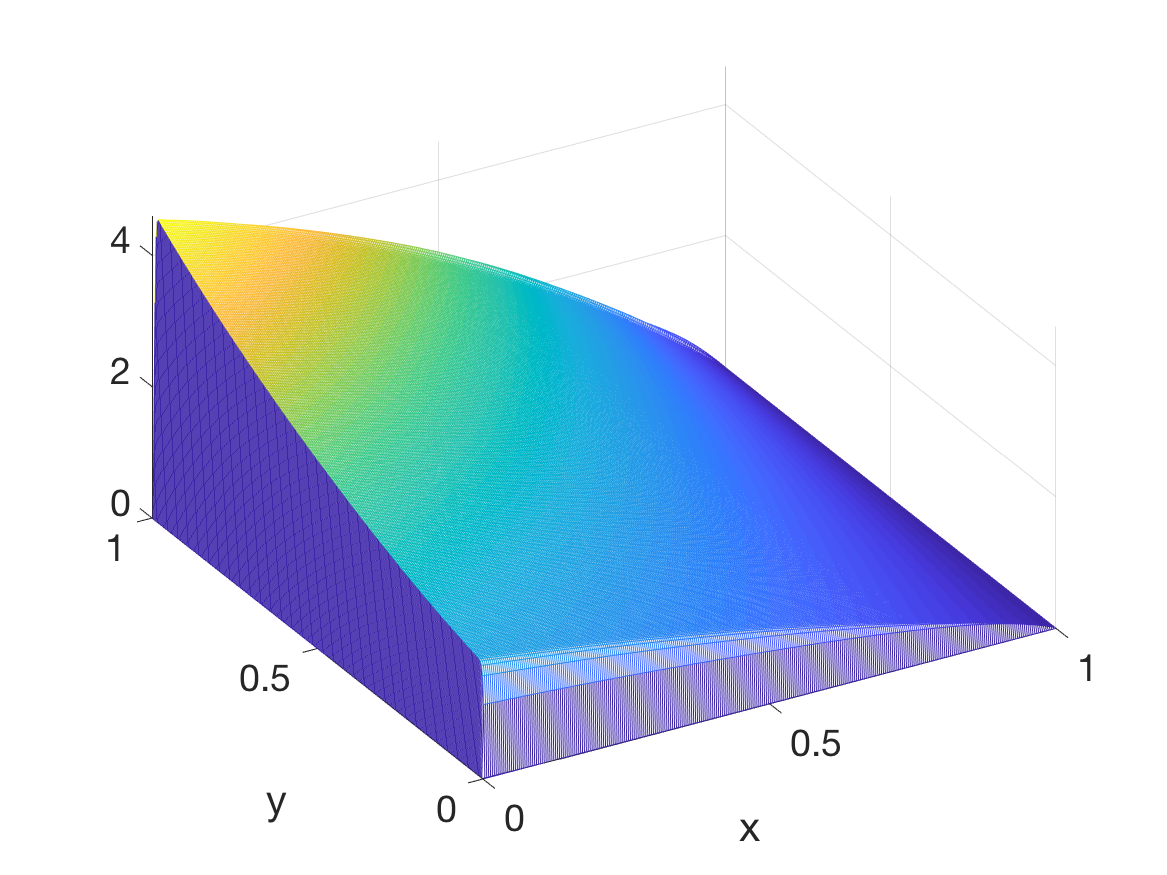}
         \caption{$f=\exp(x + y)$, $\ep= 10^{-5}$; EPOCHS=3500, $N^2 = 350^2$}
     \end{subfigure}     
     \caption{The surface plot of the neural net approximation $\widetilde{u}$ by SL-PINN (\ref{pbl_sys_SL-PINN}) for Problem (\ref{compti_f}).}
        \label{fig:plain3-1}
\end{figure} 
 
For more non-compatible examples, we consider the functions $f = \cos(\pi y)$ and $\exp(x+y)$. The neural net solutions $\widetilde{u}$ by the SL-PINN  in (\ref{pbl_sys_SL-PINN}), depicted in Figure \ref{fig:plain3-1}, agree with the numerical solutions in \cite{JT05}. To achieve higher accuracy in (B), we require a larger number of collocation points, specifically $N=350$.

 In future work, we will investigate a system similar to (\ref{pbl_sys}) for more general problems involved in various boundary layers in different locations and their interactions.
 
\section{Conclusion}

In this study, we have introduced a semi-analytic method, called SL-PINN, to enhance the numerical efficiency of 2-layer PINNs when solving singularly perturbed boundary value problems, convection-dominated equations on rectangular domains. 
For each problem, we obtained an analytic approximation of the stiff component of the solution within the boundary layer, referred to as the {\it corrector }function. Incorporating the correctors into the structure of the 2-layer PINNs, we resolved the stiffness issue of the approximate solutions and developed a new semi-analytic SL-PINNs enriched with the correctors. Through numerical simulations in Section \ref{sec_numerics}, we confirmed that our new approach produces stable and convergent approximations of the solutions.

\section*{Acknowledgments}

\noindent 
Gie was partially supported by 
Ascending Star Fellowship, Office of EVPRI, University of Louisville; 
Simons Foundation Collaboration Grant for Mathematicians; 
Research R-II Grant, Office of EVPRI, University of Louisville; 
Brain Pool Program through the National Research Foundation of Korea (NRF) (2020H1D3A2A01110658).  
Hong was supported by Basic Science Research Program through the National Research Foundation of Korea (NRF) funded by the Ministry of Education (NRF-2021R1A2C1093579) and the Korea government(MSIT)(No. 2022R1A4A3033571). 
Jung was supported by the National Research Foundation of Korea(NRF) grant
funded by the Korea government(MSIT) (No. 2023R1A2C1003120).

\end{document}